\documentclass[11pt]{article}

\usepackage[english]{babel}
\usepackage{authblk}

\usepackage{amsmath,amssymb,amsthm}
\usepackage{thmtools}
\usepackage{graphicx}
\usepackage[colorlinks=true, allcolors=blue]{hyperref}
\usepackage{tikz}
\usepackage{comment}

\usepackage[capitalise,nameinlink]{cleveref}

\declaretheorem[name=Theorem,numberwithin=section]{theo}
\newtheorem{cor}[theo]{Corollary}
\newtheorem{conj}[theo]{Conjecture}
\newtheorem{lemma}[theo]{Lemma}

\newcommand{\sm}{\setminus}

\DeclareMathOperator{\tw}{tw}
\title{(Even hole, triangle)-free graphs revisited}
\author{Beatriz Martins and Nicolas Trotignon%
\\{ENS de Lyon, CNRS, Université Lyon 1, LIP UMR 5668\\69342 Lyon Cedex 07, France}}

\begin{document}

\maketitle

\begin{abstract}
 We revisit a classical paper about (even hole, triangle)-free graphs  [Conforti, Cornu\'ejols, Kapoor and Vu\v skovi\'c, Triangle-free graphs that are signable without even holes, \textit{Journal of Graph Theory}, 34(3), 204--220, 2000]. 
 In fact, the previous study describes a more general class, the so called triangle-free odd signable graphs, and we further generalise the class to the (theta, triangle, wac)-free graphs (not worth defining in an abstract).  
 
 We exhibit a stronger structure theorem, by precisely describing basic classes and separators.  We prove that the separators preserve the treewidth and several properties.  
 Some consequences are a recognition algorithm with running time $O(|V(G)|^4|E(G)|)$, a proof that  the treewidth of graphs in the class is at most~4 (improving a previous bound of~5), and a simple criterion to decide if a graph in the class is planar. 
\end{abstract}


\section{Introduction}
\label{sec:intro}

We consider simple graphs (finite, undirected, with neither loops nor multiple edges). A graph $G$ is $H$-free if $G$ does not contain an induced subgraph isomorphic to the graph $H$.  When $L$ is a list of graphs, $L$-free means $H$-free for all $H$ in $L$.  A path in a graph from a vertex $a$ to a vertex $b$ is an \emph{$ab$-path}. The \emph{length} of a path or a cycle refers to its number of edges. A \emph{triangle} is a cycle of length~3. A \emph{hole} in a graph is a chordless cycle of length at least~4. 

A \emph{theta} is a graph made of three chordless $ab$-paths of length at least~2, vertex-disjoint apart from $a$ and $b$, and such that the only edges are the edges of the paths (note that since the paths are chordless, $a$ and $b$ are not adjacent).  
Note that the union of any two paths of a theta forms a hole. 
A \emph{wheel} $W = (H, c)$ is a graph made of a hole $H$ called the
\emph{rim} and a vertex~$c$ called the \emph{centre} with at least
three neighbours in the rim, see Figure~\ref{fig:thetaWheel}.  
A graph induced by a hole $H$ together with two adjacent vertices that both have at least three neighbours in $H$ is called a \emph{wac} for short (this stands for \emph{Wheel with Adjacent Centres}), see Figure~\ref{fig:wacs} where planar and none planar wacs are represented. 

\begin{figure}
    \centering
    \includegraphics{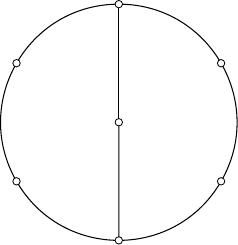}  
    \rule{2em}{0ex}
    \includegraphics{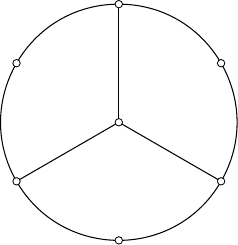}
    \caption{A theta and a wheel}
    \label{fig:thetaWheel}
\end{figure}

\begin{figure}
    \centering
    \includegraphics{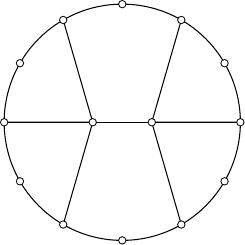}  
    \rule{2em}{0ex}
    \includegraphics{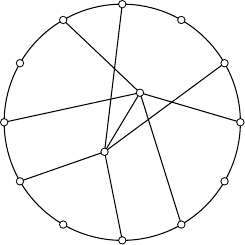}
    \caption{Two examples of (theta, triangle)-free wacs}
    \label{fig:wacs}
\end{figure}

Our main result is a structure theorem for (theta, triangle, wac)-free graphs. 

\begin{restatable}{theo}{decomp}
  \label{th:Decomp}
  If $G$ is a (theta, triangle, wac)-free graph, then $G$ is basic or $G$ has a clique separator, a proper 2-separator or a proper $P_3$-separator.
\end{restatable}

We postpone the precise definitions of the basic graphs to Section~\ref{sec:basic} and of the separators to Section~\ref{sec:separator}.  Let us just briefly explain that basic graphs include $K_1$, $K_2$ (where $K_\ell$ denotes the complete graph on $\ell$ vertices), the cube (represented in Figure~\ref{fig:cube}) and special graphs that we call \emph{daisies}, formed of one hole with paths that we call ``petals'' added around it, see Figure~\ref{fig:daisy}. 
The proper 2-separator is a separator made of two non-adjacent vertices, with the additional condition that removing it yields two components that are substantial in some sense (not reduced to a path). 
The proper $P_3$-separator is a separator equal to a path on three vertices, with a similar non-degeneracy property (plus other technical conditions).

We insist that Theorem~\ref{th:Decomp} is not only a decomposition theorem but a structure theorem for our class as we explain in Section~\ref{ssec:struct}.
An evidence of its strength is  the following corollary that is obtained by showing that our separators preserve the treewidth (see Section~\ref{ssec:tw}), which is of independent interest since they may serve to decompose other classes. 

\begin{theo}
  \label{th:TW}
  Every (theta, triangle, wac)-free graph has treewidth at most~4 (and some have treewidth exactly~4). 
\end{theo}

We also describe an algorithm that decides in time $O(|V(G)|^4|E(G)|)$ whether an input graph $G$ is (theta, triangle, wac)-free, and if so, computes a tree decomposition of width $\tw(G)$, see Theorem~\ref{th:algo}. 

Let us now motivate Theorem~\ref{th:Decomp} by explaining its relation with several well studied classes of graphs.

\subsection*{Even-hole and odd-hole-free graphs}

A hole is \emph{even} or \emph{odd} according to the parity of its length. Even-hole-free graphs attracted a lot of attention, see~\cite{vuskovic:evensurvey} for a survey. 
Despite decades of work resulting in decomposition theorems~\cite{conforti.c.k.v:eh1,dsv:ehf} and polynomial time recognition algorithms~\cite{conforti.c.k.v:eh2,DBLP:conf/stoc/LaiLT20}, no structure theorem is known for them.  
Also several questions are still open, such as the existence of a polynomial time algorithm to colour them or to find a maximum independent set. 

Still, a complete structural description is known for some subclasses of even-hole-free graphs, such as the classical \emph{chordal graphs} (that are the hole-free graphs) or the graphs where all holes have length $2k+1$ for some fixed integer $k\geq 3$, see~\cite{DBLP:journals/jctb/CookHPRSSTV24}.

Here, we are interested in the (even hole, triangle)-free graphs.  A slightly more general class was studied by Conforti, Cornu\'ejols, Kapoor and Vu\v skovi\'c in~\cite{DBLP:journals/jgt/ConfortiCKV00} that is the main inspiration of this work. 
Let us explain this. 
A wheel $(H, c)$ is \emph{even} if $c$ has an even number of neighbours in $H$. A {\em prism} is a graph made of three vertex-disjoint chordless paths
$P_1 = a_1 \dots b_1$, $P_2 = a_2 \dots b_2$, $P_3 = a_3 \dots b_3$ of length at least 1, such that $a_1a_2a_3$ and $b_1b_2b_3$ are triangles and no edges exist between the paths except those of the two
triangles.  
A  graph is \emph{odd-signable} if it contains no theta, no prism and no even wheel.  Observe that since a prism contains a triangle, a triangle-free graph is odd-signable if and only if it contains no theta and no even wheel. 
This is precisely the class studied in~\cite{DBLP:journals/jgt/ConfortiCKV00}: triangle-free odd-signable graphs, or equivalently (theta, triangle, even wheel)-free graphs (from now on, we use this last name because it is more consistent with the others classes that we consider).

It turns out that  all thetas, prisms and even wheels contain even holes, and  triangle-free wacs contain either an even wheel or a theta (and therefore even holes), see Section~\ref{ssec:ehf} for more explanations.  
So, the class of (theta, triangle, wac)-free graphs is a super-class of (theta, triangle, even wheel)-free graphs, that is in turn a super-class of (even hole, triangle)-free graphs.
These inclusions are strict, as shown by an even hole and an even wheel. Our work is therefore a generalisation of~\cite{DBLP:journals/jgt/ConfortiCKV00}.

Theorem~\ref{th:Decomp} specialises well to  the subclasses that we consider as explained in Section~\ref{ssec:ehf}: we just have to restrict the basic class.    The technicalities in our proofs are very similar to the ones in~\cite{DBLP:journals/jgt/ConfortiCKV00}, but our results differ in several ways. 
They apply to a more general class and exhibit  basic classes and separators not described in~\cite{DBLP:journals/jgt/ConfortiCKV00}.  
Still, \cite{DBLP:journals/jgt/ConfortiCKV00} contains a nice characterisation of the classes under consideration by the so-called \emph{good ear addition}, and we show in Section~\ref{ssec:ear} that it can be viewed as a corollary of Theorem~\ref{th:Decomp}, see Theorem~\ref{th:ear}. 

Odd-hole-free graphs also attracted some attention, so one may wonder if our theorems implies something for them. It turns out that (triangle, odd hole)-free graphs exactly form the class of bipartite graphs, and this is why we consider bipartite (theta, wac)-free graphs in what follows.

\subsection*{Theta-free graphs}

Since every theta contains an even hole, theta-free graphs are a generalisation of even-hole-free graphs.     
Little is known about the structure of theta-free graphs. 
A structural description is available for some subclasses, such as claw-free graphs~\cite{DBLP:conf/bcc/ChudnovskyS05} (where the \emph{claw} is the graph on four vertices, made of one vertex adjacent to three pairwise non-adjacent vertices) and (theta, wheel)-free graphs~\cite{diotRaTrVu:15,RaTrVu:partII,RaTrVu:partIII,RaTrVu:partIV}.  
Many open problems remain, such as the existence of a polynomial time algorithm for the maximum independent set problem~\cite{DBLP:journals/jacm/FaenzaOS14} or induced linkages~\cite{fialaKLP:12}
(the reference for each problem refers to a paper where the question is solved for claw-free graphs).  Also, theta-free graphs might have a quadratic $\chi$-binding function, see~\cite{Bourneuf2024On}. 

So one more time, we believe that the triangle-free case is of interest.  
In~\cite{radovanovicV:theta}, it is proved that every (theta, triangle)-free graph with no clique separator and distinct from the cube has a vertex of degree at most~2 (implying that every (theta, triangle)-free graph is 3-colourable). In~\cite{abrishamiEtAl:twIII}, it is proved that (theta, triangle)-free graphs have logarithmic treewidth.  
This suggests that (theta, triangle)-free graphs might have a simple structure, but in~\cite{DBLP:journals/jgt/SintiariT21}, it is shown that there are complex enough to  contain the so-called (theta, triangle)-free layered wheel (and in particular arbitrarily large complete graphs as a minor). 

We hope that the present work is a step toward a full understanding of the structure of (theta, triangle)-free graphs, see Section~\ref{ssec:conj}.

\subsection*{Planarity and widths}

Our structure theorem allows understanding several properties of the classes we consider. For instance, in Section~\ref{ssec:planar}, we give a simple criterion for a (theta, triangle, wac)-free graph to be planar (see~Theorem~\ref{th:planar}), and we deduce that bipartite (theta, wac)-free graphs are all planar, which is maybe not obvious from their definition.  

In Section~\ref{ssec:tw}, we prove several consequences of our structure theorem for some width parameters.  
There have been a recent interest in understanding the interplay between excluding induced subgraph and bounding the treewidth (definitions are postponed to Section~\ref{ssec:tw}), see for instance~\cite{abrishamiEtAl:twIII}. 
A full understanding of this interplay is maybe hopeless as indicated in~\cite{alecuEtAl:largeTW}, but in our particular class, we explicitly exhibit certificates for all possible values of the treewidth, see Lemma~\ref{l:tw}.  
In particular, as already stated in Theorem~\ref{th:TW}, we prove that the treewidth of any (theta, triangle, wac)-free graph is at most~4, which improves and generalises the bound of~5 proved in~\cite{DBLP:journals/dm/CameronSHV18} for (theta, triangle, even wheel)-free graphs (and therefore (even hole, triangle)-free graphs).  
Moreover, our method is potentially applicable to other classes of graphs, because we prove that the separators that we use preserve the treewidth in quite a strong sense, see Lemma~\ref{l:decTW}. 

It was observed in~\cite{DBLP:journals/dm/CameronSHV18} that bounding  the treewidth of (theta, triangle, even wheel)-free graphs by $t$ immediately yields bounds for other parameters and other classes. 
More precisely, an upper bound of $3\times 2^t$ for the cliquewidth of (theta, triangle, even wheel)-free graphs is obtained. A consequence is that for any (theta, cap, $C_4$, even wheel)-free graph, the cliquewidth is also at most $3\times 2^t$ and the treewidth at most $(t+1) \omega(G) -1$,  where $C_4$ denotes the hole of length~4, a \emph{cap} is a cycle of length at least~5 with with exactly one chord creating a triangle with the cycle and $\omega(G)$ denotes the maximum number of pairwise adjacent vertices in $G$.  
Hence our result allows to bring the cliquewidth down from 48 to 24 and the treewidth from $6\omega(G) -1$ to $5\omega(G) -1$ for the relevant classes, see Section~\ref{ssec:tw}. 

Note that in contrast, (theta, triangle)-free graphs and ($K_4$, even hole)-free graphs have arbitrarily large treewidth and cliquewidth as proved in~\cite{DBLP:journals/jgt/SintiariT21}.

\subsection*{Algorithms}

Since the treewidth of (theta, triangle, wac)-free graphs is bounded, many problems can be solved in polynomial time for them. We still describe in Section~\ref{ssec:algo} a self-contained decomposition algorithm for an input graph $G$, running in time $O(|V(G)|^4|E(G)|)$ not relying on the treewidth, and with two applications: recognition and actually building a tree representation of width $\tw(G)$ provided that the graph is (theta, triangle, wac)-free, see Theorem~\ref{th:algo}.  

Note that for the recognition algorithm, an obvious approach would be to test separately for the presence of a theta, a triangle, and a wac in the input graph.  This is tempting because detecting a theta can be performed in polynomial time, see~\cite{DBLP:conf/stoc/LaiLT20} for the fastest known algorithm. But this approach fails since we prove in Section~\ref{s:testWac} that testing for a wac is an NP-complete problem, see Theorem~\ref{th:wacNPC}.

\subsection*{Structure theorem}

As explained in Section~\ref{ssec:struct}, Theorem~\ref{th:Decomp} is not just a decomposition theorem, but rather a complete structural description of (theta, triangle, wac)-free graphs (that specialises to  (theta, triangle, even wheel)-free, (even hole, triangle)-free and bipartite (theta, wac)-free graphs), see Theorem~\ref{th:struct}.

\subsection*{Outline of the paper}

In Section~\ref{sec:basic}, we describe the basic classes. 
In Section~\ref{sec:separator}, we describe the separators. 
In Section~\ref{sec:know}, we describe several known results that we need.  
In Section~\ref{sec:lem}, we prove some technical lemmas needed in Section~\ref{sec:proofs} where we prove Theorem~\ref{th:Decomp}. 

In Section~\ref{sec:applications} we derive several consequences of our work. Namely, we explain why our classes generalises some known classes in Section~\ref{ssec:ehf}. We reprove (and improve) all results concerning the so-called ear-decomposition in Section~\ref{ssec:ear}.  We prove the planarity criterion in Section~\ref{ssec:planar}.  We give all our results concerning the treewidth and cliquewidth in Section~\ref{ssec:tw}. 
We prove that testing for a wac is NP-complete in Section~\ref{s:testWac}.  We describe the algorithms in Section~\ref{ssec:algo}. We explain why our theorem is in fact a structure theorem in Section~\ref{ssec:struct}.  We give several open questions in Section~\ref{ssec:conj}.

\subsection*{Notations}

We denote by $[k]$ the set $\{1, \dots, k\}$. When $v$ is a vertex of some graph $G$ and $X\subseteq V(G)$, we set $N_X(v) = N(v) \cap X$. 

We view a \emph{path in a graph $G$} as a sequence $P= v_1 \dots v_k$ of vertices such that for all $i, j \in [k]$, $v_iv_j\in E(G)$ if and only if $|j-i|=1$. The vertices $v_1$ and $v_k$ are the \emph{ends} of $P$, and its other vertices are \emph{internal}.   Note that this is not standard, because this is usually called a chordless, or induced path, but since all the paths that we consider are induced, this is convenient.  When $a$ and $b$ are vertices of some path $P$, we denote by  $aPb$ the subpath of $P$ with ends $a$ and $b$.  A \emph{hole} is defined as a path, with the additional conditions that $k\geq 4$ and $v_kv_1\in E(G)$. 

Since all the paper deals with induced subgraphs, we often make no difference between a set of vertices of some graph $G$ (or path or hole) and the graph it induces.  For instance, when $P$ is a path, we may write $G\sm P$ instead of $G[V(G) \sm V(P)]$.  We hope that this only make the notation less heavy and does not lead to any confusion. 

When $a, b\in V(G)$, an $ab$-path is a path whose ends are $a$ and $b$. When $H\subseteq G$, an {\it $aHb$-path} as an $ab$-path with internal vertices in $H$. 

A theta such that $a$ and $b$ are the ends of the three paths is an \emph{$ab$-theta}. When $W=(H, c)$ is a wheel, a path of $H$ of length at least~1 whose ends are adjacent to $c$ and whose internal vertices are not is called a \emph{sector} of $W$.  

We denote by  $(H, c, c')$ a wac such that $(H, c)$ and $(H, c')$ are wheels and $cc'\in E(G)$. 
Note that a wac may contain a triangle or a theta.   We sometimes need to distinguish wacs where the centres cross. A wac $W=(H, x, y)$ such that some sector of $(H, y)$ contains all neighbours of $x$ in $H$ is called a \emph{turtle}.  A wac that is not a turtle is called a \emph{crossing wac}, or \emph{c-wac} for short.

\section{Basic graphs}
\label{sec:basic}

A \emph{petal} with respect to a hole $C = c_1 \dots c_k c_1$ is a path $P = x \dots y$ disjoint from $C$ such that for some $i\in [k]$, $x$ is adjacent to $c_{i-1}$, $y$ is adjacent to $c_{i+1}$, $c_i$ is adjacent to at least one internal vertex of $P$, no edge of $c_{i-1} x P y c_{i+1}$ has both its vertices adjacent to $c_i$, and there are no other edges between $P$ and $C$ (subscripts are modulo $k$). 
The petal is \emph{centred at $c_i$} and $c_i$  the \emph{centre} of the petal.

A \emph{daisy} is a graph $G$ formed of a hole $C$ together with petals with respect to $C$ and such that no two petals have the same centre,  the centres of the petals induce a (possibly empty) subpath of $C$ or $C$ itself, and there are no other edges than the edges of the petals, the edges of the hole and the edges between some petal and the hole. 
An example is represented in Figure~\ref{fig:daisy} (this graph is represented at the end of~\cite{DBLP:journals/jgt/ConfortiCKV00} as an example of a non-planar (even hole, triangle)-free graph).   
Observe that a daisy with no petal is a hole, and a daisy with one petal is a wheel.  
A daisy is \emph{even} or \emph{odd} according to the parity of $C$. 
If every vertex of $C$ is the centre of some petal, we call $G$ a \emph{full daisy}. When $C$ has length $k$, the daisy is a $k$-daisy.  It is easy to check that every daisy is (theta, triangle, wac)-free.

\begin{figure}
    \centering
    \includegraphics{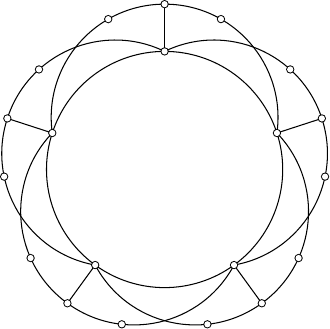}
    \caption{The Vu\v skovi\'c graph (the smallest non-planar daisy)}
    \label{fig:daisy}
\end{figure}

We  denote by $Q_i$ the path of $C$ from $c_{i-1}$ to $c_{i+1}$
that does not contain~$c_i$.  We denote by $P_i$ the petal of the daisy centred at $c_i$ (if any), and set $P_i = x_i \dots y_i$ where $x_i$ is adjacent to $c_{i-1}$ and $y_i$
is adjacent to $c_{i+1}$.  
We observe that $P_i$ and $Q_i$ form a hole $H_i$, and $W_i=(H_i, c_i)$ is a wheel.  
We observe that $Q_i$ is a sector of $W_i$ and the other sectors are subpaths of $c_{i-1} x_i P_i y_i c_{i+1}$ called the \emph{external sectors of $W_i$}.

We already noted that a wheel $W = (H, c)$ is a daisy with exactly one petal.  
To see this, consider a sector $S= s\dots s'$ of $W$ and observe that $H\sm S$ is a petal with respect to the hole $C = csSs'c$. 
Hence, there are several ways (one for each sector) to view the wheel as a daisy, but whatever sector is chosen, $c$ is the centre of the unique petal. 
When there are more petals (or none), the hole of the daisy is uniquely determined: it is the unique hole going though all centres of its petals.  
In all cases, the vertices that are centres of wheels are all on the hole $C$.

The cube is the graph on eight vertices represented on Figure~\ref{fig:cube}.  A graph is \emph{basic} if it is isomorphic to $K_1$, $K_2$, the cube or a daisy. 

\begin{figure}
    \centering
    \includegraphics{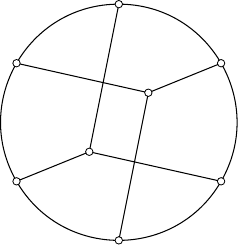}
    \caption{The cube}
    \label{fig:cube}
\end{figure}

\section{Separators}
\label{sec:separator}

We here define all the separators that we need in our structure
theorem.  We call \emph{separator} of $G$ a set $X$ of vertices of a graph $G$ such that $G\sm X$ has at least two connected components. In what follows $X$ will be a (possibly empty) clique, a set $\{a, b\}$ of two non-adjacent vertices, or a set inducing a path $acb$ of length~2.  Since we are interested in triangle-free graphs, the clique separator is on at most two vertices, and the possible separators are precisely the induced subgraphs of some path of length~2.  There might therefore be a more unified way to view all of them, but we believe this would lead to unnecessary abstraction and complicated explanations, so we prefer viewing them one by one.   

\subsection*{Clique separator}

A \emph{clique separator} in a graph $G$ is a (possibly empty) clique $K$ such that $G\sm K$ has at least two connected components.  

Note that basic graphs, thetas, triangle, wacs, even holes, prisms, even wheels and odd holes have no clique separators. 

When $K$ is a clique separator of a graph $G$ and $X_1$, \dots, $X_k$ are the connected components of $G\sm K$, we call \emph{blocks of decomposition} of $G$ with respect to $K$ the graphs $G[K \cup X_1]$, \dots, $G[K \cup X_k]$.

\begin{lemma}
  \label{l:decClique}
  When $G$ is a graph with a clique separator $K$, then $G$ is
  theta-free (resp.\ triangle-free, wac-free, turtle-free, c-wac-free, even-hole-free, prism-free, even-wheel-free, bipartite) if and only if all its blocks of decomposition with respect to $K$ are theta-free (resp.\ triangle-free, wac-free, turtle-free, c-wac-free, even-hole-free, prism-free, even-wheel-free, bipartite).
\end{lemma}

\begin{proof}
  If $G$ is theta-free (resp.\ triangle-free, wac-free, turtle-free, c-wac-free, even-hole-free, prism-free, even-wheel-free, bipartite) then its block of decomposition are also theta-free (resp.\ triangle-free, wac-free, turtle-free, c-wac-free, even-hole-free, prism-free, even-wheel-free, bipartite) since they are induced subgraphs of $G$.

  Let us now consider all the converse claims. We view ``bipartite'' as ``(triangle, odd hole)-free'', so each claim is about excluding some structures. Each time, we assume that some structure $H$ under consideration is in $G$ and in none of the blocks of decomposition. This leads to a contradiction, because $H$ has no clique separator as we already observed, so if it is present in $G$, it must lie entirely in one block of decomposition. 
\end{proof}

A graph is \emph{atomic} if it has no clique separator.

\subsection*{2-separator}

When $a$ and $b$ are non-adjacent vertices of some graph $G$ such that $G\sm \{a, b\}$ is not connected, we call $\{a, b\}$ a \emph{2-separator} of $G$.

\begin{lemma}
  \label{l:eligible2K1}
  If $\{a, b\}$ is a 2-separator of an atomic theta-free graph~$G$, then  $G\sm \{a, b\}$ has exactly two connected components $X$ and $Y$, and $G$ contains an $aXb$-path and an $aYb$-path. 
\end{lemma}

\begin{proof}
  Let $X$ and $Y$ be connected components of $G\sm\{a,b\}$. There exists  an $aXb$-path (resp.\ an $aYb$-path) in $G$ for otherwise $a$ or $b$ is a clique-separator of $G$, a contradiction to $G$ being atomic. 
  If $G\sm\{a,b\}$ has a third component~$Z$, then an  $aXb$-path, an $aYb$-path and an $aZb$-path form an $ab$-theta, a contradiction.
\end{proof}

In view of Lemma~\ref{l:eligible2K1}, a 2-separator $\{a, b\}$ of some graph $G$ is \emph{proper} if $G\sm \{a, b\}$ has exactly two components $X$ and $Y$ and $G[X \cup \{a, b\}]$ (resp.\ $G[\{a, b\} \cup Y]$) contains an $ab$-path $P$ (resp.\ $Q$) but is not equal to an $ab$-path, that is $V(P) \subsetneq X \cup \{a, b\}$ (resp.\ $V(Q) \subsetneq Y \cup \{a, b\}$).  We then call $(X, \{a, b\}, Y)$ a \emph{split of $G$ with respect to the proper 2-separator $\{a, b\}$}.

It is easy to check that basic graphs (see Lemma~\ref{l:decDaisy} below), thetas, triangles, wacs, even holes, prisms, even wheels and odd holes have no proper 2-separators (apart from triangles, they may have 2-separators, but for all such separators $\{a, b\}$, there exists at least one component $X$ of $G \sm \{a, b\}$ such that $G[X \cup \{a, b\}]$ is the unique $aXb$-path). 

A (theta, triangle, wac)-free graph $G$ with a proper 2-separator $\{a, b\}$ is represented on the left in Figure~\ref{fig:separators}. It is easy see that $G$ is not basic, has no clique separator, no proper $P_3$-separator (to be defined), and that $\{a, b\}$ is the unique proper 2-separator of~$G$.

\begin{figure}
    \centering
    \includegraphics{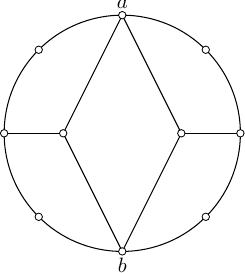}
    \rule{2em}{0ex}
    \includegraphics{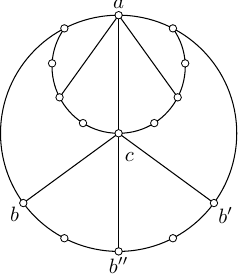}
    \caption{Graphs with separators}
    \label{fig:separators}
\end{figure}

We call \emph{blocks of decomposition} of $G$ with respect to a proper 2-separator the two graphs $G_X = G[X \cup \{a, b\} \cup Q]$ and $G_Y = G[P \cup \{a, b\} \cup Y ]$ where $Q$ (resp.\ $P$) is an $aXb$-path  (resp. $aYb$-path).
Note that $G_X$ (resp.\ $G_Y$) is not uniquely defined, it may depend on the choice of the path $Q$ (resp.\ $P$). Note that $G_X$ and $G_Y$ have both fewer vertices than $G$, because the 2-separator is proper.

\begin{lemma}
  \label{l:decProper2sepAtomic}
  When $(X, \{a, b\}, Y)$ is a split with respect to a proper 2-separator of a graph $G$, then $G$ is atomic if and only if both $G_X$ and $G_Y$ are atomic.
\end{lemma}

\begin{proof}
  Suppose first that $G$ is non-atomic.   
  So there exists  a clique separator $K$ of $G$. 
  Up to symmetry, $K\subseteq X\cup \{a\}$.  
  If two connected components of $G\sm K$ have a non-empty intersection with $Y\cup \{b\}$, then $\{a\}$ is a clique separator of $G_Y$.  
  Otherwise, $Y\cup \{b\}$ is included in some connected component of $G\sm K$.   
  It follows that $K$ is a clique separator of $G_X$.  
  So, one of  $G_X$ or $G_Y$ is non-atomic.   

  Suppose conversely that one of $G_X$ or $G_Y$ is non-atomic, say $G_X$ up to symmetry.  
  So there exists a clique separator $K$ of $G_X$.  
  If $K\cap X = \emptyset$, then one of $\{a\}$ or $\{b\}$ is a clique separator of $G$. 
  Otherwise, up to symmetry, $K\subseteq X\cup \{a\}$ and $Q\sm a$ is included in some connected component of $G\sm K$.   
  It follows that $K$ is a clique separator of $G$, so $G$ is non-atomic.
\end{proof}

\begin{lemma}
 \label{l:decProper2sep}
 When $(X, \{a, b\}, Y)$ is a split with respect to a proper
 2-separator of a graph $G$, then $G$ is theta-free (resp.\ triangle-free, wac-free, turtle-free, c-wac-free, even-hole-free, prism-free, even-wheel-free, bipartite) if and only if both $G_X$ and $G_Y$ are theta-free (resp.\ triangle-free, wac-free, turtle-free, c-wac-free, even-hole-free, prism-free, even-wheel-free, bipartite).
\end{lemma}

\begin{proof}
  If $G$ is theta-free (resp.\ triangle-free, wac-free, turtle-free, c-wac-free, even-hole-free, prism-free, even-wheel-free, bipartite) then $G_X$ and $G_Y$ are theta-free (resp.\ triangle-free, wac-free, turtle-free, c-wac-free, even-hole-free, prism-free, even-wheel-free, bipartite) since they are induced subgraphs of $G$.

  Let us now consider the converse claims.  We view ``bipartite'' as ``(triangle, odd hole)-free'', so each claim is about excluding some structure. For each claim, we assume that both $G_X$ and $G_Y$ do not contain the object $H$ under consideration, suppose that $G$ does, and look for a contradiction. 

  Since $H$ has no clique-separator and no proper 2-separator, $H$ is in $G$ and in none of $G_X$ and $G_Y$, $\{a, b\}$ must be a (non proper) 2-separator of~$H$. 
  In particular, $H$ cannot be a triangle. Up to the symmetry between $X$ and~$Y$, we may therefore assume that $H\cap (Y\cup \{a, b\})$ is an $aYb$-path~$R$.  Hence replacing  $R$ by $Q$ in $H$ yields a theta (resp.\ wac, turtle, c-wac, hole, prism, even-wheel) whenever $H$ is a theta (resp.\ wac, turtle, c-wac, hole, prism, even-wheel) of $G_X$.
  
  To obtain a contradiction in all cases, it remain to check that $Q$ and $R$ have the same parity for the case of even and odd holes.  This holds because if their parity differ, then the union of one of them with $P$ yields a hole of $G_X$ with the forbidden parity. 
\end{proof}

A graph is \emph{superatomic} if it is atomic and has no proper 2-separator.

\subsection*{$P_3$-separator}

When $acb$ is a path in some graph $G$ such that $G\sm acb$ is not connected, $acb$ is a \emph{$P_3$-separator} of $G$.

\begin{lemma}
  \label{l:eligibleP3}
   If a  theta-free graph $G$ contains a $P_3$-separator $acb$, then $G\sm acb$ has exactly two connected components $X$ and $Y$ such that $G$ contains an $aXb$-path and an $aYb$-path, and every $aYb$-path of $G$ contains at least one internal vertex adjacent to~$c$. Moreover, if none of $X\cup \{a, b\}$ and $Y\cup \{a, b\}$ induces an $ab$-path, then  $c$ has neighbours in both $X$ and~$Y$.
\end{lemma}

\begin{proof}
  Let $X$ and $Y$ be components of $G\sm acb$. Since $G$ is atomic, both $a$ and $b$ must have a neighbour in $X$ for otherwise $G$ has a clique separator (contained in $\{a,~c,~b\}$).  Hence $G[X\cup\{a,b\}]$ is connected and there exists an $aXb$-path in $G$. Similarly, there exists an $aYb$-path in $G$.
  
  If $G\sm acb$ has a third connected components $Z$, then an $aXb$-path, an $aYb$-path and an $aZb$-path would form a theta of $G$, a contradiction. 
  
  Now suppose that exists  an $aXb$-path $P$ and an $aYb$-path $Q$ path that both have no internal vertex adjacent to $c$. Then, $P$, $Q$ and $acb$ form a theta, a contradiction. So, up to a swap of the names $X$ and $Y$, every $aYb$-path has at least one internal vertex adjacent to $c$.
  
  If $c$ has no neighbours in $X$, then $\{a, b\}$ is a  2-separator of~$G$, which is proper (because of the assumption that none of $X\cup \{a, b\}$ and $Y\cup \{a, b\}$ is an $ab$-path), a  contradiction. 
\end{proof}

In view of Lemma~\ref{l:eligibleP3}, when $acb$ is a $P_3$-separator, and $X$ is a connected component of $G\sm acb$, we say that $X$ is \emph{loose} if there exists an $aXb$-path with no internal vertex adjacent to $c$, and \emph{tight} otherwise.
A $P_3$-separator $acb$ of some graph $G$ is \emph{proper} if $G\sm acb$ has exactly two connected components $X$ and $Y$,  there exists an $aXb$-path and an $aYb$-path in $G$, $c$ has neighbours in both $X$ and $Y$, $X$ is loose, $Y$ is tight, and none of $G[X \cup \{a, b\}]$ and $G[Y \cup \{a, b\}]$ is an $ab$-path.  We then call $(X, acb, Y)$ a \emph{split of $G$ with respect to the $P_3$-separator $acb$}.

Note that Lemma~\ref{l:eligibleP3} does not guaranty that $X$ is loose, but the separators that we will find in proofs will still turn out to be proper.  
A graph~$G$ with no proper 2-separator and a proper $P_3$-separator $acb$ (or $acb'$) is represented on the right in Figure~\ref{fig:separators}.  Observe that $acb''$ is not a proper $P_3$-separator because no component of $G\sm acb''$ is loose. 

The triangle clearly has no $P_3$-separator. An $ab$-theta might have a $P_3$-separator, but only if one of the $ab$-paths has length~2, and then the separator is a path $acb$ of the theta, and it is not proper.  It is easy to check that all $P_3$-separators of a wac $W= (H, c, c')$ are of the form $acb$ (or $ac'b$) where $a, b\in H$, and one component of $W\sm acb$ (or $W\sm ac'b$) is a subpath of some sector of $(H, c')$ (or $(H, c)$).  So a wac has no proper $P_3$-separator. 

We call \emph{blocks of decomposition} of $G$ with respect to a proper $P_3$-separator the two graphs $G_X = G[X \cup acb \cup Q]$ and $G_Y = G[P \cup acb \cup Y ]$ where $P$ is an $aXb$-path with no internal vertex adjacent to $c$, and $Q$ is any $aYb$-path.

\begin{lemma}
 \label{l:decProperP3SepAtomic} 
  When $(X, acb, Y)$ is a split with respect to a proper $P_3$-separator of a graph $G$, then $G$ is atomic if and only if both $G_X$ and $G_Y$ are atomic.
\end{lemma}

\begin{proof}
  Suppose first that $G$ is non-atomic.   
  So there exists  a clique separator $K$ of $G$. 
  Up to symmetry, either $K\subseteq X \cup \{a, c\}$ or $K\subseteq Y \cup \{a, c\}$.  
  First consider the case when $K\subseteq X \cup \{a, c\}$. 
  If two connected components of $G\sm K$ have a non-empty intersection with $Y\cup \{b\}$, then $\{a, c\}$ is a clique separator of $G_Y$.  
  Otherwise, $Y\cup \{b\}$ is included in some connected component of $G\sm K$.   
  It follows that $K$ is a clique separator of $G_X$.  So, one of  $G_X$ or $G_Y$ is non-atomic.   
  The case when $K\subseteq Y \cup \{a, c\}$ is similar (though not formally symmetric, since $X$ is loose and $Y$ is tight).
  
  Suppose now that  $G_X$ is non-atomic.  
  So there exists a clique separator $K$ of $G_X$.  
  If $K\cap X = \emptyset$, then one of $\{a, c\}$ or $\{c, b\}$ is a clique separator of~$G$. 
  Otherwise, up to symmetry, $K\subseteq X\cup \{a, c\}$ and $Q\sm a$ is included in some connected component $Z$ of $G_X\sm K$, and there is another component $Z'$. It follows that $K$ is a clique separator of~$G$, separating $Z'$ from $Y\sm \{a, c\}$, so $G$ is non-atomic.
  
  The case when $G_Y$ is non-atomic is similar.  
\end{proof}

\begin{lemma}
 \label{l:decProperP3Sep-2Sep} 
  When $(X, acb, Y)$ is a split with respect to a proper $P_3$-separator of a graph $G$, then $G$ is superatomic if and only if both $G_X$ and $G_Y$ are superatomic.
\end{lemma}

\begin{proof}
  By Lemma~\ref{l:decProperP3SepAtomic}, it is enough to prove the following statement under the assumption that $G$, $G_X$ and $G_Y$ are atomic: $G$ has a proper 2-separator if and only if one of $G_X$ or $G_Y$ has a proper 2-separator.

  Suppose first that $G$ has a proper 2-separator $\{u, v\}$.  If $u\in X$ and $v\in Y$, then some component $C$ of $G\sm \{u, v\}$ is disjoint from $acb$, so $C\subseteq X$ or $C\subseteq Y$.  It follows that $u$ or $v$ is a clique separator of $G$, a contradiction to $G$ being atomic.  So, either $\{u, v\} \subseteq X \cup acb$ or $\{u, v\} \subseteq  acb \cup Y$. 
  
  If $\{u, v\} = \{a, b\}$, then since $c$ has neighbours in both $X$ and $Y$, the component $C$ of $G\sm \{u, v\}$ that contains $c$ overlaps $X$ and $Y$. The other component $D$ must be in $X$ or $Y$, and accordingly $\{a, b\}$ is a proper separator of $G_X$ or $G_Y$. Up to symmetry, we may assume from here on that $b\notin \{u, v\}$.

  Either $\{u, v\} \subseteq X \cup \{a, c\}$ or $\{u, v\} \subseteq Y \cup \{a, c\}$.  
  First consider the case when $\{u, v\}\subseteq X \cup \{a, c\}$. 
  If two connected components of $G\sm \{u, v\}$ have a non-empty intersection with $Y\cup \{b\}$, then $\{a, c\}$ is a  clique separator of $G_Y$, a contradiction.  
  Otherwise, $Y\cup \{b\}$ is included in some connected component of $G\sm \{u, v\}$.   
  It follows that $\{u, v\}$ is a proper 2-separator of~$G_X$.  
  The case when $ \{u, v\} \subseteq Y \cup \{a, c\}$ is similar (though not formally symmetric, since $X$ is loose and $Y$ is tight).

  Suppose now that  $G_X$ has a proper 2-separator $\{u, v\}$.  Since by definition of proper $P_3$ separators $G[X]$ is connected, $u$ and $v$ cannot both be in the marker path $Q$. 
  So, if $\{u, v\} \cap X = \emptyset$, then one of $\{a, c\}$ or $\{c, b\}$ is a clique separator of $G$, a contradiction to $G$ being atomic. 
  Otherwise, up to symmetry, $\{u, v\}\subseteq X\cup acb$ and the interior of $Q$ is included in some connected component of $G\sm \{u, v\}$. It follows that $\{u, v\}$ is a proper 2-separator of $G$.
  
  The case when $G_Y$ has a proper-2-separator is similar.  
\end{proof}

\begin{lemma}
  \label{l:decProperP3Sep}
  When $(X, acb, Y)$ is a split with respect to a proper $P_3$-separator of a graph $G$, then $G$ is theta-free (resp.\ triangle-free, wac-free, turtle-free, c-wac-free, even-hole-free, prism-free, even-wheel-free, bipartite) if and only if both $G_X$ and $G_Y$ are theta-free (resp.\ triangle-free, wac-free, turtle-free, c-wac-free, even-hole-free, prism-free, even-wheel-free, bipartite).
\end{lemma}

\begin{proof}
  If $G$ is theta-free (resp.\ triangle-free, wac-free, turtle-free, c-wac-free, even-hole-free, prism-free, even-wheel-free, bipartite) then $G_X$ and $G_Y$ are theta-free (resp.\ triangle-free, wac-free, turtle-free, c-wac-free, even-hole-free, prism-free, even-wheel-free, bipartite) since they are induced subgraphs of $G$.

  Let us now consider the converse claims.  Recall that we view ``bipartite'' as ``(triangle, odd hole)-free'', so each claim is about excluding some structure. For each claim, we assume that both $G_X$ and $G_Y$ do not contain the object $H$ under consideration, suppose that $G$ does, and look for a contradiction. 

  As observed above, $H$ has no clique-separator, no proper 2-separator and no proper $P_3$-separator.  But $H$ is in $G$ and in none of $G_X$ and $G_Y$.  Hence,  $\{a, b\}$ is a (non proper) 2-separator of $H$, or $acb$ is a (non proper) $P_3$-separator of $H$. In particular, $H$ cannot be a triangle.  We may therefore assume that $H\cap (X\cup \{a, b\})$ is an $aXb$-path, or $H\cap (Y\cup \{a, b\})$ is an $aYb$-path. In both cases, we call this path $R$. Hence replacing $R$ by $P$ or $Q$ in $H$ yields a theta (resp.\ wac, turtle, c-wac, hole, prism, wheel) whenever $H$ is a theta (resp.\ wac, turtle, c-wac, hole, prism, wheel) of $G_X$.  For the claim about wheels, it is worth noting that $P$ and $Q$ form the rim of a wheel of $G$ (resp.\ $G_X$, $G_Y$) centred at $c$, so in fact the equivalence holds because all graphs under consideration contain wheels.  
  
  To obtain a contradiction in all cases, it remain to check that the parity is preserved for holes and wheels. 
  
  Let us first check this for holes. Suppose first that $R$ is an $aYb$-path. We need to check  that $Q$ and $R$ have the same parity.  This holds because if their parity differ, then the union of one of them with $P$ yields a hole of $G_Y$ with the forbidden parity. The case when  $R$ is an $aXb$-path is similar ($P$ and $R$ have the same parity, or a hole of the forbidden parity exists in $G_X$).
  
  Finally let us check the parity of wheels. Suppose first that $R$ is an $aYb$-path. We need to check  that $Q$ and $R$ have the same $c$-parity, where the \emph{$c$-parity} of a path is the parity of the number of neighbours of $c$ that it contains.  This holds because if their $c$-parity differ, then the union of one of them with $P$ yields an even wheel of $G_Y$. The case when  $R$ is an $aXb$-path is similar ($P$ and $R$ have the same $c$-parity, or an even wheel exists in~$G_X$).
\end{proof}

\subsection*{Daisies and separators}

\begin{lemma}
  \label{l:decDaisy}
  A daisy has no clique separator, no proper 2-separator and no proper $P_3$-separator. 
\end{lemma}

\begin{proof}
  The claim about clique separators is obvious.  The 2-separators $\{a, b\}$ of some daisy $G$ are all such that $a$ and $b$ are in some external sector of some wheel of $G$, or are the ends of some subpath of $C$ containing no centre of a petal. In either case, $\{a, b\}$ is not proper.  The $P_3$-separators of some daisy $G$ are the $P_3$'s $a c_i b$ where $a$ and $b$ are any neighbours of some $c_i$ that is the centre of some petal.  They are all non-proper because for one component $X$ of $G\sm a c_i b$, $G[X\cup \{a, b\}]$ is  an $ab$-path. 
\end{proof}

\section{Known results}
\label{sec:know}

All the results in this section are already known. We include their proofs for the sake of completeness. The following is implicit in~\cite{confortiCKV97}.

\begin{theo}
\label{cor:us}
  If $G$ is (theta, triangle, wheel)-free graph, then $G$ is $K_1$, $K_2$, a hole or $G$ has a clique separator.  
\end{theo}

\begin{proof}
 We may assume that $G$ is an atomic graph. If $G$ has no hole then, since $G$ is triangle-free, $G$ must be a forest. Hence $G$ must be $K_1$ or $K_2$, since they are the only atomic forests.

Now, let $C$ be a hole contained in $G$. We may assume that $G$ is not a hole. Since $G$ is atomic, there is a path $R=x\ldots y$ in a component $Z\subseteq (G\setminus C)$ and there are non-adjacent vertices $u,v\in V(C)$ such that $xu,yv\in E(G)$. Assume that $u,v$ and $R$ are chosen to minimise the length of $R$.

Since $G$ is triangle-free, $x\neq y$. Otherwise $R$ is just a vertex and $C\cup R$ induces either a theta or a wheel. Furthermore, $N_C(x)=u$ (resp.\ $N_C(y)=v$), otherwise $V(C)\cup\{x\}$ (resp.\ $V(C)\cup \{y\}$) induces either a theta or a wheel.

If there is a vertex $w\in V(C)\setminus\{u,v\}$ such that $N_R(w)\neq \emptyset$, then $w$ is adjacent to a vertex $z$ in the interior of $R$. Notice that this contradicts the minimality of $R$ if $w$ is a non-neighbour of either $v$ or $u$. Hence, $w$ must be adjacent to both $u$ and $v$, which implies that at most two vertices of $C$ can have neighbours in internal vertices of $R$. If only one vertex $w\in V(C)\setminus\{u,v\}$ has a neighbour in $R$, then $G[C\cup R]$ induces a wheel centred in $w$. So suppose that $w_1,w_2\in V(C)\setminus\{u,v\}$ have neighbours in the interior of $R$ say $r_1$ and $r_2$ respectively. Notice that in this case $C=uw_1vw_2u$. The path $r_1Rr_2$ has a smaller length than $R$ and there are two non-adjacent vertices $w_1,w_2\in V(C)$ such that $w_1r_1,w_2r_2\in E(G)$, which contradicts the minimality of $R$.
\end{proof}

\begin{lemma}
  \label{l:vAtachHole}
  Suppose that $G$ is a (theta, triangle)-free graph and $H$ is a hole  of $G$.  If $v \in G\sm H$ has at least two neighbours in $H$, then  $(H, v)$ is a wheel.  In particular, no vertex in $G \sm H$ has exactly two neighbours in $H$. 
\end{lemma}

\begin{proof}
If $v$ has at least two neighbours in $H$ and $V(H)\cup \{v\}$ is not a wheel, then $|N_H(v)|= 2$. So let $u,w$ be the neighbours of $v$ in $H$. Since $G$ is triangle-free, $uw\notin E(G)$. Hence $V(H)\cup\{v\}$ induces an $uw$-theta.
\end{proof}

The following is proved in~\cite{radovanovicV:theta}.

\begin{lemma}
  \label{l:decCube}
  If $G$ is a (theta, triangle)-free graph that contains the cube, then $G$ is isomorphic to the cube or $G$ has a clique separator.  
\end{lemma}

\begin{proof}
  We denote the vertices of a cube $H$ contained in $G$ with a bipartition $A= \{a_1, \dots, a_4\}$ and $B= \{b_1, \dots b_4\}$, in such a way that $a_i b_j\in E(H)$ if and only if $i\neq j$. If some vertex $v\in G\sm H$ has at least two neighbours in $A$, say $a_1$ and $a_2$ up to symmetry, then $\{v, a_1, a_2, b_3, b_4\}$ contains a triangle or induces an $a_1a_2$-theta, a contradiction.  Hence, $v$ has at most one neighbour in $A$, and symmetrically in $B$.  If $v$ has two neighbours in $H$, they must therefore be $a_1$ and $b_1$ up to symmetry, so $\{v, a_1, a_2, a_3, b_1, b_2, b_3\}$ induces an $a_1b_1$-theta, a contradiction.  Hence, a vertex $v\in G\sm H$ has at most one neighbour in $H$. 
  
  We may assume that $G\sm H$ has a connected component $Z$ such that $N_H(Z)$ is not a clique.  So, there exists a path $R = x\dots y$ in $Z$ and non-adjacent $u, v\in H$ such that $xu, yv\in E(G)$. We choose $u,~v$ and $R$ subject to the minimality of $R$. This implies that if an internal vertex  of $R$ has a neighbour $w$ in $H$, then $w$ must be a common neighbour of $u$ and $v$.   
  
  If $u\in A$ and $v\in B$, then up to symmetry, $u=a_1$ and $v=b_1$.  Note that  $u$ and $v$ have no common neighbours in $H$, so no internal vertex of $R$ has a neighbour in $H$.  Hence, $R$, $a_1$, $a_2$, $a_3$, $b_1$, $b_2$ and $b_3$ form an $a_1b_1$-theta, a contradiction. 
  We may therefore assume that $u=a_1$ and $v=a_2$.  No internal vertices of $R$ is adjacent to $a_3$, $a_4$,  $b_1$ and $b_2$ because they are not common neighbours of $a_1$ and $a_2$. Hence, $R$,  $a_1$, $a_2$, $a_3$, $a_4$, $b_1$ and $b_2$ form an $b_1b_2$-theta, a contradiction. 
\end{proof}

Variants of the following are proved in \cite[Theorem 2.5]{DBLP:journals/jgt/ConfortiCKV00} for (theta, triangle, even wheel)-free graphs and in~\cite[Lemma 3.1]{DBLP:journals/corr/abs-2001-01607}. 
\begin{lemma}
  \label{l:VattachW}
  Suppose that $G$ is a (theta, triangle, c-wac, cube)-free graph and $W = (H, c)$ is a wheel of $G$.  If $v$ is a vertex of $G\sm W$ then $N_W(v) \subseteq S\cup \{c\}$ where $S$ is a  sector of $W$.  Moreover, if $G$ is wac-free, then either $N_W(v) \subseteq S$ or $N_W(v) \subseteq \{c\}$.  
\end{lemma}

\begin{proof}
 If $v$ has at most one neighbour in $H$ then, since $G$ is (theta,triangle)-free, either $N_W(v) \subseteq \{c\}$, or $N_W(c)$ is included in some sector of $W$. Hence, by Lemma~\ref{l:vAtachHole}, we may assume that $|N_H(v)|\geq 3$. So $W' = (H, v)$ is a wheel. 
 
 If $vc\in E(G)$, then $N_W(v) \subseteq S\cup \{c\}$ where $S$ is a  sector of $W$, for otherwise $(H, c, v)$ is a c-wac.  Note also that $(H, c, v)$ is a wac (specifically a turtle), so when $G$ is wac-free, there is a contradiction.  We may therefore assume from here on that  $vc\notin E(G)$.  
 
 Suppose by contradiction that $N_W(v)$ is not included in some sector of~$W$. 
 So there exist distinct sectors $S=s\ldots s'$ and $T=t\ldots t'$ of $W$ such that $v$ has at least one neighbour in $S\setminus T$ and at least one neighbour in $T\setminus S$.  Up to symmetry, we assume that $s$, $s'$, $t$ and $t'$ appear in this order along $H$ and that $s\neq t'$ (but possibly $s'=t$). 

 If $|N_{S}(v)|\geq 2$, then let $x$ (resp.\ $x'$) be the neighbour of $v$ in $S$ closest to $s$ (resp.\ $s'$) along $S$. 
 Thus, $R= vx'Ss'csSxv$ is a hole in $G$. Moreover, $N_{T\setminus S}(v)\neq\emptyset$, so let $y$ (resp.\ $y'$) be the neighbour of $v$ in $T$ closest to $t$ (resp.\ $t'$) along $T$. Notice that $R'=vy'Tt'c$ is a chordless path in $G$ and $G[R\cup R']$ is a $vc$-theta in $G$, unless $t=s'$ and $ty'\in E(G)$, in which case, $R$ and $y'$ form a $vt$-theta, a contradiction.
 
Thus, $v$ can have at most one neighbour in each sector of $W$, and similarly $c$ can have at most one neighbour in each sector of $W'$. Moreover, $N_H(c)\cap N_H(v)=\emptyset$, for otherwise either $c$ has at least two neighbours in some sector of $(H, v)$ or $v$ has at least two neighbours in some sector of $(H, c)$ (see Fig.~\ref{fig:lem_attach}).

\begin{figure}[!h]
    \centering
    \includegraphics{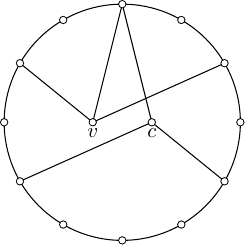}
    \rule{1cm}{0cm}
    \includegraphics{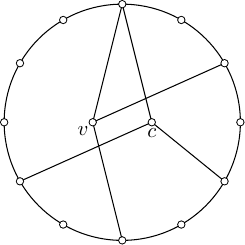}
    \caption{When $v$ and $c$ have a common neighbour}
    \label{fig:lem_attach}
\end{figure}

Hence, $|N_H(c)|=|N_H(v)|$ and the neighbours of $c$ and $v$ alternate in~$H$.  
Consider three consecutive sectors  of $W= (H, c)$: $Q_1 = q_1\dots q_2$, $Q_2= q_2\dots q_3$ and $Q_3 = q_3 \dots q_4$. 
Since the neighbours of $c$ and $v$ alternate, $v$ has a unique neighbour $v_i$ in each $Q_i$. If $v_3q_1\notin E(G)$, then $vv_2Q_2q_3$, $vv_3Q_3q_3$ and $vv_1Q_1q_1cq_3$ form a theta, a contradiction. 
Hence, $v_3q_1\in E(G)$ (in particular $q_4=q_1$ and $W$ has three sectors).  
A symmetric argument shows that $v_iq_{i+1}$ and $q_iv_{i+1}$ are edges for all $i\in [3]$ (subscripts are modulo 3).  
Hence $H$ and $v$ form a cube, a contradiction. 
\end{proof}

\section{Lemmas}
\label{sec:lem}

\begin{lemma}
  \label{l:CattachW}
  Suppose that $G$ is a (theta, triangle, wac, cube)-free graph and $W = (H, c)$ is a wheel of $G$.  
  If $Z$ is a connected induced subgraph of $G\sm W$, then $N_H(Z)$ is included in some sector of $W$.
  \end{lemma}

\begin{proof}
  Otherwise, there exist distinct sectors $S = s\dots r$ and $T = t \dots r'$ of~$W$ such that some vertex of $S\sm T$ and some vertex of $T \sm S$ both have neighbours in $Z$.  
  So, $Z$ contains a path $P = x\dots y$ such that $x$ has neighbours in $S\sm T$ and $y$ has neighbours in $T \sm S$. 
  We suppose that $S$, $T$ and $P$ are chosen subject to the minimality of $P$.  
  By Lemma~\ref{l:VattachW}, $P$ has length at least~1 (so $x\neq y$), $N_W(x) \subseteq S$ and $N_W(y) \subseteq T$ (in particular none of $x$ and $y$ is adjacent to $c$).  
  By the minimality of $P$, internal vertices of $P$ have no neighbours in $H$, except possibly the unique vertex in $S \cap T$ (if any).  
  Note that internal vertices of $P$ may be adjacent to $c$.

  Let $x'$ (resp.\ $x''$) be the neighbour of $x$ in $S$ closest to  $s$ (resp.\ to $r$) along~$S$.  
  Let $y'$ (resp.\ $y''$) be the neighbour of $y$ in $T$ closest to $t$ (resp.\ to~$r'$) along $T$.
  These names are given in such a way that $s$, $x'$, $x''$, $r$, $r'$, $y''$, $y'$ and $t$ appear in this order along $H$.  
  Up to symmetry, we assume that $s\neq t$ (implying that $st\notin E(G)$ since $G$ is triangle-free, but possibly $r=r'$).

  Suppose that some internal vertex of $P$ has a neighbour in $H$.
  Then, by the minimality of $P$, we must have $r=r'$ and $r$ is the only vertex of $H$ with neighbours in the interior of $P$. 
  Consider the hole $H' = x' (H \sm r) y' y P x x'$.  
  By Lemma~\ref{l:vAtachHole}, $(H', c)$ is a wheel (because $c$ has at least two neighbours in~$H'$: $s$ and $t$).  
  So, $r$ has exactly one neighbour $z$ in $H'$ (that is in the interior of $P$), for otherwise by Lemma~\ref{l:vAtachHole} $(H', c, r)$ is a wac.    
  Now, $zrc$, a shortest path from $z$ to $c$ in $z P x x' S s c$, and a shortest path from $z$ to $c$ in $z P y y' T t c$ form a $zc$-theta. 
  Hence, no internal vertex of $P$ has a neighbour in $H$. 
  
  If $x'=x''$, then $x' S s c$, $x' S r c$ and a shortest path from  $x'$ to $c$ in $x' x P y y' T t c$ form an $x'c$-theta, unless $c$ has no neighbour in $P$,  $r=r'$, $y'=y''$ and $y'r\in E(G)$.  
  In particular, $y'=y''$ and a symmetric
  proof based on this fact shows that $x'r\in E(G)$.  Hence, $P$ and
  $H$ form an $x'y'$-theta, a contradiction. 
  Hence, $x'\neq x''$.  Symmetrically, $y'\neq y''$.

  Now, removing the interior of $x' S x''$ and $y' T y''$  in $P\cup H$ yields an $xy$-theta, unless $xy\in E(G)$.  
  In this last case, by Lemma~\ref{l:vAtachHole}, $x$ and $y$ are both centres of a wheel with rim $H$, so $(H, x, y)$ is a wac, a contradiction again. 
\end{proof}

Lemma~\ref{l:CattachW} has an interesting consequence regarding
$P_3$-separators (but we never use it in the rest of the paper).

\begin{lemma}
  \label{l:looseSide}
  Let $G$ be a (theta, triangle, wac, cube)-free graph and $(X, acb, Y)$ be a split with respect to a proper $P_3$-separator of $G$. 
  Then, no internal vertex of any $aXb$-path is adjacent to $c$.
\end{lemma}

\begin{proof}
  Recall that from the definition of splits and proper
  $P_3$-separators, $X$ is loose and $Y$ is tight, that is there exists an $aXb$-path $P$  with  no  internal vertex adjacent to $c$, and every
  $aYb$-path that has least one internal vertex adjacent to $c$.  Consider such an $aYb$-path $Q$.  Suppose for a contradiction that there exists an $aXb$-path $R$  that does contain an internal vertex adjacent to $c$.  Note that $R$, $Q$ and $c$ form a wheel $W$ centred at $c$, and no sector of $W$ contain both $a$ and $b$.

  Because of the interior of $P$, there exists two sectors $S$
  and $S'$ of $W$, and some path $T = t\dots t'$ that contains neither $c$ nor neighbours of $c$ and such that $t$ has a neighbour in $S\sm S'$ and $t'$ has a neighbour in $S'\sm S$.  We suppose $S$, $S'$ and $T$ are chosen subject to the minimality of $T$.

  By the minimality of $T$, $T$ is disjoint from $W$.  It therefore contradicts Lemma~\ref{l:CattachW}. 
\end{proof}

The next lemma describes how a connected set of vertices may attach to the centre and the rim of a wheel.  
Before stating it, let us describe two examples represented in Fig.~\ref{f:attach}.  
In both examples, a path $P = x \dots y$ attaches to a wheel $W = (H, c)$, and both examples turn out to be daisies. In the first one, the hole is $c s S s' c$ and there are two petals: $P$ and $H\sm S$.  
In the second one, the hole is $c s y t' S s' c$ and there are three petals: the interior of $t' S s$, $P\sm y$ and $H \sm S$.
The next lemma shows that in some sense, these two examples are the only possibilities.

\begin{figure}
  \center
  \includegraphics{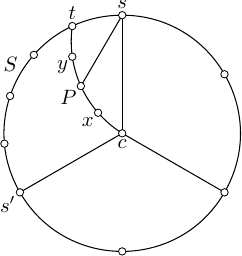}
  \rule{2em}{0ex}
  \includegraphics{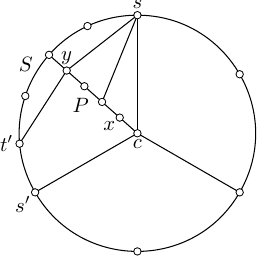}
  \caption{Two examples of attachments to a wheel\label{f:attach}}
\end{figure}

\begin{lemma}
  \label{l:CattachPetal}
  Suppose that $G$ is a (theta, triangle, wac, cube)-free graph and $W = (H, c)$ is a wheel of $G$.  
  Let $Z$ be a connected induced subgraph of $G\sm W$ such that $c$ has a neighbour in~$Z$ and $N_W(Z)$ is not a clique.  
  Then $Z$ contains a path $P = x \dots y$ of length at least~1 and $W$ has a sector $S = s \dots s'$ such that $N_W(x)= \{c\}$, $N_W(y) \subseteq S$, $y$ has at least one neighbour in the interior of~$S$, $s$ has at least one neighbour in the interior of $P$ and no internal vertex of $P$ has a neighbour in~$W\sm s$. 
  Moreover, if $y$ has a neighbour in~$S$ that is not adjacent to $s$, then $ys\in E(G)$.  
  In all cases, $W$ and $P$ form a daisy with two or three petals. 
\end{lemma}

\begin{proof}
  By Lemma~\ref{l:CattachW}, there exists a sector $S = s \dots s'$ of
  $W$ such that $N_H(Z)\subseteq S$.  Suppose first that no internal
  vertex of $S$ has a neighbour in~$Z$.  Then, since $N_W(Z)$ is not a
  clique, we have $N_W(Z) = \{s, c, s'\}$.  Hence, $H$ and a shortest
  path of $Z$ from a neighbour of $s$ to a neighbour of $s'$ form an
  $ss'$-theta, a contradition.

  So let $P = x \dots y$ be a shortest path in~$Z$ such that
  $xc \in E(G)$ and $y$ has a neighbour in the interior of $S$. By
  Lemma~\ref{l:VattachW}, $x\neq y$.  Let $t$ (resp.\ $t'$) be the
  neighbour of $y$ in~$S$ that is closest to $s$ (resp.\ $s'$) along
  $S$.  If both $s$ and $s'$ have neighbours in~$P\sm y$, then $H$ and
  a shortest path of $P \sm y$ from a neighbour of $s$ to a neighbour of
  $s'$ in~$P\sm y$ form an $ss'$-theta, a contradiction.
  Hence, up to symmetry, we may assume that $s'$ has no neighbour in
  $P\sm y$.

  If $s$ has no neighbour in the interior of $P$, then
  $S \cup P \cup \{c\}$ contains a $cy$-theta (if
  $t\neq t'$) or a $ct$-theta (if $t=t'$).  So, $s$ has at
  least one neighbour in the interior of $P$.

  So, $P$ has all the claimed properties. 
  It remains to prove that  if $y$ has a neighbour in~$S$ that is not adjacent to $s$, then $ys\in E(G)$. 
  Suppose not, that is $t's \notin E(G)$ and $t\neq s$.
  Hence, $y t S s$, $y t' S s' c s$ and a shortest path from $y$ to $s$ in~$P\cup \{s\}$ form a $ys$-theta if $t\neq t'$, or a $ts$-theta if $t=t'$, in both cases a contradiction.  
 
  The fact that $W$ and $P$ form a daisy with two or three petals has been explained in the paragraph before the statement.
\end{proof}

\begin{lemma}
  \label{l:twoPetals}
  Suppose that $G$ is an atomic (theta, triangle, wac, cube)-free
  graph and $D$ is a daisy of $G$ with at least two petals.  If
  some induced connected set $Z\subseteq G \sm D$ has neighbours in some petal
  of $D$, then $G$ has a proper 2-separator or a proper
  $P_3$-separator.
\end{lemma}

\begin{proof}
  Since $G$ is atomic, by Lemmas~\ref{l:eligible2K1}
  and~\ref{l:eligibleP3}, to check that a 2-separator or a
  $P_3$-separator is proper, only the condition about
  $G[X\cup \{a, b\}]$ and $G[Y \cup \{a, b\}]$ not being an $ab$-path,
  and the condition about the loose component of a $P_3$-separator
  have to be checked.

  Suppose that some vertex of $Z$ has a neighbour in some petal $P_i$
  of $D$.

  We claim that there exists an external sector $S = s\dots s'$ of
  $W_i = (H_i, c_i)$ such that $N_{D}(Z)$ is included in $S\cup \{c_i\}$. Indeed, by
  Lemma~\ref{l:CattachW}, $N_{H_i}(D)$ is included in some sector $S$
  of $W_i$, and since some vertex of $Z$ has neighbours in $P_i$, $S$
  must be an external sector of $W_i$.  Let us check that $N_D(Z)$ is
  included in $S\cup \{c_i\}$.  So suppose for a contradiction that
  $Z$ contains a vertex with a neighbour $x\in D \sm (S\cup \{c_i\})$.
  From the definition of $S$, $x$ is not in $W_i$ (in particular,
  $x\notin C$).  So, $x\in P_j$ for some petal $P_j$, $j\neq i$.  Now
  $Z\cup P_j$ is connected, so it contradicts Lemma~\ref{l:CattachW}
  applied to $W_i$, because $P_j$ has a neighbour in
  $C \sm \{c_{i-1}, c_{i}, c_{i+1}\}$, so a neighbour in $H_i \sm S$.

  From the claim that we just proved, we see that $\{s, s'\}$ is a
  proper 2-separator of $G$ (if $c_i$ has no neighbour in $Z$) or
  $\{s, c_i, s'\}$ is a proper $P_3$-separator of $G$ (the component
  containing $Z$ is loose because of the path $S$).  Note that the
  condition that none of $X\cup \{s, s'\}$ and $Y\cup \{s, s'\}$ is a
  path is satisfied.  This is because of $Z$ for the component that
  contains $Z$, and because of the second petal for the other side
  (this is only the place where we need the second petal).
\end{proof}

\section{Proof of Theorem~\ref{th:Decomp}}
\label{sec:proofs}

\decomp*

\begin{proof}
  We may assume that $G$ is atomic.  By Lemma~\ref{l:decCube}, we may assume that $G$ is cube-free.  By Theorem~\ref{cor:us}, we may assume that $G$ contains a wheel. So $G$ contains a daisy with at least one petal, and we consider such a daisy $D$ of $G$ that is maximal (in the sense of inclusion for the vertex-set).  We may assume that $G\sm D$ has a connected component $Z$ for otherwise $G$ is basic.

To describe $D$, we use notations as in the definition of a daisy.
Since $D$ has at least one petal, we assume up to
symmetry that $c_1$, \dots, $c_\ell$ where $1 \leq \ell \leq k$ are the centres of the petals of $D$.

Suppose first that some vertex of $Z$ has a neighbour in some petal $P_i$ of~$D$.  By Lemma~\ref{l:twoPetals}, we may then assume that $D$ has a unique petal, so it is the wheel $W_1 = (H_1, c_1)$.  
By Lemma~\ref{l:CattachW}, $N_{H_1}(Z)$ is included
in some sector $S = s \dots s'$ of $W_1$.  If $N_{D}(Z) \subseteq S$, then $\{s, s'\}$ is a proper 2-separator of $G$.  Otherwise, $Z$ contains a neighbour of $c_1$, and by Lemma~\ref{l:CattachPetal},
either $N_{D}(Z)$ is a clique (so $G$ has a clique separator, a contradiction), or $D$ can be extended to a daisy with two or three petals, a contradiction to the maximality of $D$.

From here on, we may assume that no vertex of $Z$ has a neighbour in
some petal $P_i$ of $D$.  If no centre of a petal has a neighbour in $Z$,
since $G$ is atomic, then some vertex of $C$ that is not the centre of
a petal has a neighbour in $Z$ (in particular, $\ell < k$).  So,
$\{c_{\ell + 1}, c_k\}$ is a clique separator (a contradiction), or a
proper 2-separator.  Hence, we may assume that some centre $c_i$ of
some petal $P_i$ has a neighbour in $Z$.

Let us apply Lemma~\ref{l:CattachPetal} to $Z$ and $W_i$. If
$N_{W_i}(Z)$ is a clique, then so is $N_D(Z)$ (because
$N_D(Z) = N_{W_i}(Z)$), and $G$ has a clique separator, a
contradiction.  So, $Z$ contains a path $P = x \dots y$ of length at
least~1, $xc_i\in E(G)$, and $y$ has a neighbour in the interior of
$Q_i$ (recall that $Q_i$ is the sector of $W_i$ contained in $C$).
Moreover, some end of $Q_i$, say $c_{i+1}$ up 
to symmetry, has neighbours in the interior
of $P$.

If the only neighbour of $y$ in $Q_i$ is $c_{i+2}$, then $D$ has no
petal centred at $c_{i+1}$, for otherwise, such a petal $P_{i+1}$
together with $P$ and $C\sm c_{i+1}$ would form a theta from $c_i$ to
$c_{i+2}$.  It follows that $i=\ell$, $D$ has no petal centred at
$c_{\ell +1}$ (in particular, $\ell < k$).  Hence adding $P$ to $D$
yields a daisy that contradicts the maximality of $D$.

Hence, $y$ has a neighbour in $Q_i$ that is not adjacent to $c_{i+1}$ (so it is not $c_{i+2}$). By Lemma~\ref{l:CattachPetal},
$yc_{i+1} \in E(G)$.  We denote by $Q_{i+1}$ the path $C\sm c_{i+1}$.
Let $c_{j}$ be the neighbour of $y$ in that is closest to $c_i$ along
the path $Q_{i+1}$. Note that $j \neq i$ since $x \neq y$.  We now
consider a daisy $D'$ defined as follows.  The hole is
$C' = c_i c_{i+1} y c_j Q_{i+1} c_i$.  The petals are $P\sm y$,
$c_{i+2} Q_{i+1} c_{j-1}$, and all petals of $D$ centred at some
vertex of $c_{j+1} Q_{i+1} c_i$ (in particular $P_i$).  If all petals
of $D$ are petals of $D'$, then $D'$ contradicts the maximality of
$D$.  Hence, $D$ has at least one petal $R$ that is not a petal of
$D'$, and that is therefore centred at some vertex in
$c_{i+1} Q_i c_j$.  This petal $R$ is therefore a connected
subset of $G\sm D'$ that has a neighbour in some petal of $D'$.  Hence,
by Lemma~\ref{l:twoPetals} applied to $D'$, $G$ has a proper
2-separator or a proper $P_3$-separator.
\end{proof}

\section{Consequences}
\label{sec:applications}

We here collect several consequences of Theorem~\ref{th:Decomp}.

\subsection{Even-hole-free, even-wheel-free and bipartite graphs}
\label{ssec:ehf}

We here explain why our results specialise to a structure theorem for (theta, triangle, even wheel)-free graphs, (even hole, triangle)-free graphs and bipartite (theta, wac)-free graphs.  The following proves the inclusions of classes claimed in the introduction. 

\begin{lemma}
  \label{l:oddS}
  The cube, every theta and every even wheel contains an even hole, and every triangle-free wac contains either an even wheel or a theta (and therefore an even hole).
\end{lemma}

\begin{proof}
  The cube obviously contains $C_6$.  In a theta, two amongst the three paths have the same parity and their union is therefore an even hole.  If an  even wheel $W=(H, c)$ is even-hole-free, then  all  sectors $S$ of $W$ have odd length for otherwise $cSc$ has even length.  Therefore, the length of $H$, that is the sum of the lengths of the sectors of $W$, is even, a contradiction. 
  
  Let $W = (H, c, c')$ be a triangle-free wac. If $W$ contains no even wheel, then $c$ has an even number of neighbours in each sector $S$ of $(H, c')$, for otherwise $(c'Sc', c)$ is an even wheel or a theta. Moreover, these neighbours are in the interior of $S$ since $G$ is triangle-free.   Hence, the number of neighbours of $c$ in $H$ is the sum of even numbers, so $(H, c)$ is an even wheel, a contradiction.
\end{proof}

The following lemma states that even-wheel-free, even-hole-free and bipartite daisies are easily described.  
Hence, it is fair to consider Theorem~\ref{th:Decomp} as a decomposition theorem for (theta, triangle, even wheel)-free graphs (just replace daisies by even-wheel-free daisies in the definition of the basic class),  as a decomposition theorem for (even-hole, triangle)-free graphs (just remove the cube and replace daisies by even-hole-free daisies in the definition of the basic class), and as a decomposition theorem for bipartite (theta, wac)-free graphs (just replace daisies by bipartite daisies in the definition of the basic class).

\begin{lemma}
  \label{l:daisyEHF}
  A daisy is even-wheel-free if and only if every centre of some petal has an odd number of neighbours in the  petal. 

  It is even-hole-free if and only if it is even-wheel-free, based on an odd hole and every external sector has odd length.
  
  It is bipartite if and only if it is based on an even hole and every external sector has even length. 
\end{lemma}

\begin{proof}
  We use the notations as in the definition of a daisy $G$. 
  The wheels of $G$ are centred on some vertex $c_i$, and has the same number of sectors as the wheels $W_i$, so the first statement follows.  
  
  The holes of $G$ are either the hole $C$ with some edge-disjoint $P_3$'s $c_{i-1} c_i c_{i+1}$ replaced by some petal, or some hole $c_iSc_i$ where $S$ is some external sector of some  wheel $W_i$, so the second statement follows.
  
  The third statement is clear. 
\end{proof}

\subsection{Ear decomposition} 
\label{ssec:ear}

One of the main result in~\cite{DBLP:journals/jgt/ConfortiCKV00} is a
construction of all (theta, triangle, even wheel)-free graphs (and in particular of all (even hole, triangle)-free graphs) by a ``sequence of good ear additions''.  
Let us reprove this result and first define the notions that we need.  
A $P_3$, say $acb$, in a graph $G$ is \emph{good} if for every  $ab$-path $P$ in $G\sm c$,  no internal vertex of $P$ is adjacent to $c$.  

The following lemma gives a characterisation of good $P_3$'s.  
Note that Condition~\ref{i:hole} is essential as shown in Figure~\ref{fig:notGood}. 
A daisy $G$ with two petals centred at $c_1$ and $c_2$ is represented.  
There is a non-good $P_3$:  $a c_1 b$. However, it satisfies Conditions~\ref{i:center} and~\ref{i:spoke} of the lemma. Also, it can be checked that adding an ear with respect to $ac_1b$ (see the definition below) creates a $bc_2$-theta. 

\begin{figure}
    \centering
    \includegraphics{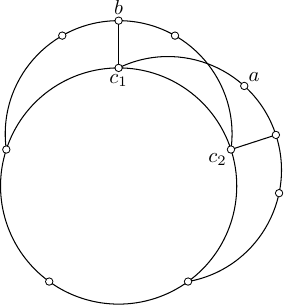}
    \caption{A not so good $P_3$}
    \label{fig:notGood}
\end{figure}

\begin{lemma}
  \label{l:goodEquiv}
  Let $G$ be an atomic (theta, triangle)-free graph.
  A path $acb$ of~$G$ is a good $P_3$ if and only if it satisfies the following three conditions:    
  \begin{enumerate}
  \item\label{i:hole} some hole of $G$ goes through $a$, $c$ and $b$, 
  \item\label{i:center}  there is no wheel $(H, c)$ in $G$ such that $a, b \in H$ and
  \item\label{i:spoke} there is no wheel $(H, c')$ in $G$ such that $a, c, b \in H$ and $cc'\in E(G)$.
  \end{enumerate}
\end{lemma}

\begin{proof}
  Suppose first that $acb$ is good.  Since $G$ is atomic, there exists a path~$P$ from $a$ to $b$ in $G\sm c$, and this path has no internal vertex adjacent to $c$ since $acb$ is good. Hence, $acb$ and $P$ form a hole, so Condition~\ref{i:hole} is satisfied.  Conditions~\ref{i:center} and~\ref{i:spoke} are clearly satisfied for otherwise, there exists a path from $a$ to $b$ with internal vertices adjacent to $c$, a contradiction to $acb$ being good. 
  
  Let us prove the converse. So, suppose for a contradiction $acb$ is not good while Conditions~\ref{i:hole}, \ref{i:center} and~\ref{i:spoke} are satisfied. 
  So there exists an $ab$-path~$P$, not containing $c$ and containing at least one internal vertex adjacent to $c$.  We choose $P$ subject to the minimality of the  number $k\geq 1$ of internal vertices of $P$ adjacent to $c$.   
  
  Call \emph{sector} of $P$ any subpath of $P$ of length at least~1, whose ends are adjacent to $c$ and whose internal vertices are not (so $P$ has at least two sectors). 
  
  By Condition~\ref{i:hole}, some hole goes through $acb$. So there exists an $ab$-path~$Q$ in $G\sm c$ such that no internal vertex of $Q$ is adjacent to $c$. 
  
  Not all vertices of $Q$ are in $P$ because the interior of $Q$ contains no neighbour of $c$. 
  Since $a$ and $b$ are in distinct sectors of $P$, there exists a subpath $Q'= a'\dots b'$ of $Q$ such that $a'$ and $b'$ are vertices of distinct sectors $S'$ and $T'$ of $P$, and the interior of $Q'$ is disjoint from $P$. Note that $a'$ (resp.\ $b'$) is either one of $a$ or $b$, or is an internal vertex of $S'$ (resp.\ $T'$). 

  Consider a subpath $R= x \dots y$ of $Q'$ and two distinct sectors $S = s\dots r$ and $T = r'\dots t$ of $P$ such that $x$ has neighbours in $S\sm T$, $y$ has neighbours in $T\sm S$ and no sector of $P$ contains $N_P(x) \cup N_P(y)$. 
  The existence of $R$ comes from the fact that the interior of $Q'$  satisfies these conditions.  
  Suppose that $S$, $T$ and $R$ are chosen subject to the minimality of $R$ (possibly $x=y$). 
  
  Up to symmetry, we suppose that $a$, $s$, $r$, $r'$, $t$ and $b$ appear in this order along $P$.  
  Let $x'$ (resp.\ $x''$) be the neighbour of $x$ in $S$ closest to $s$ (resp.\ to $r$) along $S$.  Let $y'$ (resp.\ $y''$) be the neighbour of $y$ in $T$ closest to $r'$ (resp.\ to $t$) along $T$. 
  
  Since $R$ is a subpath of the interior of $Q'$, it contains neither $c$ nor neighbours of $c$ and $R$ is disjoint from $P$.  Also, by the minimality of $R$, an internal vertex of $R$ has no neighbours in $P$, except possibly $r$ if $r=r'$.  Finally, no internal vertex of $r P r'$ has a neighbour in $R$,  except possibly if $x=y$.  
 
  Consider the path $Z = a P x' x R y y'' P b$.  
  We claim that some neighbour of~$c$  in $rPr'$ is not in $Z$.  
  Otherwise $r=r'$ or $rPr'$ is a sector of $P$, and $N_P(x) \cup N_P(y)$ is included in some sector of $P$, a contradiction to the definition of~$R$. 
  Hence, $Z$  contains less neighbours of $c$ than~$P$.  So, the interior of $Z$ contains no neighbour of $c$ for otherwise, $Z$ contradicts the minimality of $k$.  
  This implies that $S$ is the sector of $P$ that contains $a$ (so $s=a$), $T$ is the sector of $P$ that contains $b$ (so $t=b$), $x' \neq  r$ and $y'' \neq  r'$. 

  Suppose $x'\neq x''$.  
  Then $x x' P a c$, $x x'' P r c$ and $x R y y'' P b c$ form an $xc$-theta unless $r'=r$, $y'=y''$ and $y'r \in E(G)$.  
  But then, $y' r c$, $y' P b c$ and $y'y R x x' P a c$ form a $y'c$-theta unless $x''=r$. 
  In this last case, $(aPx'xRyy'Pbca, r)$ is a wheel, a contradiction to Condition~\ref{i:spoke} (note that in this case, $P$, $R$ and $c$ form a daisy with two petals). 
  This proves that $x'=x''$, and symmetrically $y'=y''$ (recall that $x' \neq  r$ and $y'' \neq  r'$).
  
  Suppose that $x'\neq a$.  
  Then, the three paths $x' P a c$, $x' P r c$ and $x' x R y y' P b c$ form a theta unless $r = r'$ and $y'r \in E(G)$.  In this last case, $y'\neq b$, so a symmetric argument proves that $x'r\in E(G)$.  
  Then, the hole $a P x' x R y y' P b c a$ is the rim of wheel centred at $r$, a contradiction to Condition~\ref{i:spoke}.  Hence, $x'=a$, and symmetrically, we can prove $y'=b$. 
  
  Suppose that $r=r'$.  If $r$ has some neighbours in $R$, then $rca$, $rPa$ and a shortest path from $r$ to $a$ in $\{r, a\} \cup R$ form a theta, a contradiction. So, $r$ has no neighbours in $R$.  
  Hence, $axRybPa$ is a hole, and therefore the rim of a wheel centred at $c$.  
  This is a contradiction to Condition~\ref{i:center}.  So, $r\neq r'$.  By the minimality of $R$, it follows that no internal vertex of $R$ has a neighbour in $P$.   
  
  Suppose that some vertex of $rPr'$  has some neighbour in $R$.  By the minimality of $R$, this implies that $R$ has length~0, so $x=y$.  Then, $xac$, $ybc$ and some shortest path from $x$ to $c$ in $\{x, c\} \cup rPr'$ form an $xc$-theta, a contradiction. Hence, no vertex of $rPr'$ has some neighbour in $R$. 
  Hence, $axRybPa$ is a hole,  and therefore the rim of a wheel centred at $c$.  This is a contradiction to Condition~\ref{i:center}. 
\end{proof}

Adding an \emph{ear} to a path $acb$ in a graph $G$ means adding a path $P= x\dots y$ disjoint from $G$ and such that $xa, yb \in E(G)$, $c$ has at least one neighbour in the interior of $P$, no two consecutive vertices of $axPyb$ are adjacent to $c$ and there are no other edges.  
A \emph{good ear addition} means an ear addition made to a good~$P_3$.

Note that our definition slightly differs from the one in~\cite{DBLP:journals/jgt/ConfortiCKV00}.  
First, in~\cite{DBLP:journals/jgt/ConfortiCKV00}, good ear additions are defined with the three conditions of Lemma~\ref{l:goodEquiv}.  
We believe that our definition is simpler (both to state and to use, see below how we use it) and slightly more general (with our definition, adding an ear to a $acb$ is permitted even when $c$ is a clique separator separating $a$ from $b$). 
Another difference is a condition in~\cite{DBLP:journals/jgt/ConfortiCKV00} that we do not need, about the parity of the number of neighbours of $c$ in the ear, to prevent creating an even wheel when adding an ear (recall that~\cite{DBLP:journals/jgt/ConfortiCKV00} is concerned with even-wheel-free graphs).
Note finally that we need Lemma~\ref{l:goodEquiv} only to explain that our notion of good ear addition is (almost) equivalent to the one in~\cite{DBLP:journals/jgt/ConfortiCKV00}, otherwise we never use it.

\begin{lemma}
  \label{l:earPreserve}
  Suppose that $G'$ is obtained from $G$ by a good ear addition. 
  Then $G'$ is theta-free (resp.\ triangle-free, wac-free, turtle-free, c-wac-free, prism-free) if and only if $G$ is theta-free (resp.\ triangle-free, wac-free, turtle-free, c-wac-free, prism-free).
\end{lemma}

\begin{proof}
  If $G'$ is theta-free (resp.\ triangle-free, wac-free, turtle-free, c-wac-free, prism-free), then so is  $G$ since $G$ is an induced subgraph of $G'$.  
  Let us prove the converse for all structure under consideration. 
  It is clear for triangles. 
  
  If $G'$ contains a theta $H$, then either $H$ is a theta of $G$, or $H$ contain the ear~$P$ (it must contain it entirely since a theta has no clique separator).  
  Note that $c$ cannot be in $H$, hence replacing $P$ by $c$ yields a theta of $G$ (because $acb$ is good, so there is no neighbour of $c$ in $H \sm \{a, b\}$).   
  In all cases, $G$ contains a theta.
  The proof for prisms is similar. 
  
  If $G'$ contains a wac $H$, we may assume that $c$ is a centre of $H$, for otherwise the same proof as above works. 
  Let $c'$ be the other centre. Some path from $a$ to $b$ through $c'$ contradicts $acb$ being a good $P_3$ of $G$, because $cc'\in E(G)$.  The same proofs works for turtles and c-wacs.  
\end{proof}

\begin{lemma}
  \label{l:findOneEar}
  Every atomic (theta, triangle, wac)-free graph on at least three vertices and distinct from a hole and a cube can be obtained from an atomic graph by a good ear addition.
\end{lemma}

\begin{proof}
  Since every  every wheel is obtained from a hole by a good ear addition, it is enough to prove by induction the following statement. 
  
  \medskip
  
  \noindent ($\star$): Every (theta, triangle, wac)-free graph distinct from a hole and a wheel contains two disjoint paths $S$ and $T$ such that $G\sm S$ (resp.\ $G\sm T$) is atomic and $G$ is obtained from $G\sm S$ (resp.\ $G\sm T$) by adding the good ear $S$ (resp.~$T$).
  
  \medskip
  
  The statement $(\star)$ is clear for daisies, because a daisy distinct from a hole and a wheel must have at least two petals, and any two of them can serve as the two ears that we need. 
  So, by Theorem~\ref{th:Decomp} and since $G$ is atomic, it is enough to consider the following cases. 
  
  \medskip
  
  \noindent{\bf Case~1:} $G$ has a proper 2-separator with split $(X, \{a, b\}, Y)$.

  We claim that $X$ contains a path $S$ such that $G_X \sm S$ is atomic, $G_X$ is obtained from $G_X \sm S$ by adding the good ear $S$ to some good $P_3$ of $G_X$. Note that $G_X$ is not a hole (because $\{a, b\}$ is proper). 
  
  If $G_X$ is a wheel $W = (H, c)$, then some sector $R = x \dots y$ of $W$ contains the marker path $Q$. 
  We then set $S = H\sm R$, and $G_X$ is obtained from $xRycx$ by adding the ear $S$ to $xcy$. 
  If $G_X$ is not a wheel, then by the induction hypothesis, $G_X$ contains two disjoint ears $S$ and $T$ as in $(\star)$.  
  Since they cannot both overlap $Q$, one of them, say $S$ up to symmetry, is in $X$. This proves our claim.  For simplicity, we also denote by $xcy$ the good $P_3$ on which $S$ is added. 
  
  We now prove that $S$ is in fact such that  $G\sm S$ is atomic and $G$ is obtained from $G\sm S$ by adding the good ear $S$ to some good $P_3$ of $G$.  Since $G_X \sm S$ is atomic, so is $G\sm S$ (we omit the details, they are similar to the ones in Lemma~\ref{l:decProper2sepAtomic}). The path $xcy$ is good also in $G$ because a path from $x$ to $y$ in $G$ has all its internal vertices either in $G_X$ or in $Y$, and in either case, they are non-adjacent to $c$ since $xcy$ is good in $G_X$.  This proves our claim. 
  
  Now we observe that the argument can be mirrored in $Y$ to obtain the second ear $T$ that we need, thus proving $(\star)$ for $G$.  
  
  \medskip
  
  \noindent{\bf Case~2:} $G$ has a proper $P_3$-separator with split $(X, acb, Y)$.

  We claim that $X$ contains a path $S$ such that $G_X \sm S$ is atomic, $G_X$ is obtained from $G_X \sm S$ by adding the good ear $S$ to some good $P_3$ of $G_X$. Note that $G_X$ is neither a hole nor a wheel (because $acb$ is proper). 
  So, we may apply directly the induction hypothesis and obtain two disjoint ear $S$ and $T$. 

  The proof is then very similar to the previous case, but we need to be more careful about proving that $S$ and $T$ cannot both overlap the marker path $Q$. If they do, then $S$ is an ear added to a path $scr$ and $T$ to a path $pcr$ where $r$ is some internal vertex of $Q$.  But then there is a contradiction, because $ac$ is a clique separator of either $G_X \sm S$ or $G_X \sm T$ (separating $r$ from $X$), contradicting $G_X \sm S$ and $G_X \sm T$ being atomic.
  
  Mirroring the proof in $G_Y$ is easy, because the internal vertices of the marker path $P$ have degree~2, so not both ears obtained by induction can overlap $P$.  
\end{proof}

We are now ready to prove the main theorem of this section, that is very similar to Theorem~6.4 from~\cite{DBLP:journals/jgt/ConfortiCKV00}.

\begin{theo}
  \label{th:ear}
 Let $G$ be an atomic graph on at least three vertices and distinct from the cube. Then, $G$ is (theta, triangle, wac)-free if and only if $G$ can be obtained from a hole by a sequence of good ear additions.   
\end{theo}

\begin{proof}
  If $G$ is (theta, triangle, wac)-free, then it is obtained from some hole by a sequence of good ear additions by repeated applications of  Lemma~\ref{l:findOneEar}. The converse if obtained by repeated applications of Lemma~\ref{l:earPreserve}.
\end{proof}

It is easy to obtain variants of Theorem~\ref{th:ear} for the classes of (theta, triangle, even wheel)-free graphs, (even hole, triangle)-free graphs and bipartite (theta, wac)-free graphs.  Just ensure that adding the ear is class-preserving.  We omit the description, it is very similar to the descriptions of restrictions to the definition of petals in Lemma~\ref{l:daisyEHF}.

Note that Theorem~\ref{th:ear} and it variants are maybe not so convenient, in particular to design recognition algorithms, because testing if a $P_3$ is good is difficult, as certified by the following theorem. 

\begin{theo}
  \label{th:goodNPC}
  Testing if a prescribed path $acb$ of an input graph $G$ is a good $P_3$ is a coNP-complete problem.  
\end{theo}

\begin{proof}
  First note that the problem is in coNP because an $ab$-path in $G\sm c$ containing an internal vertex adjacent to $c$ certifies that $acb$ is not good. 
  
  Bienstock~\cite{bienstock:evenpair} proved that deciding if an input graph $G$ with two prescribed vertices $c$ and $c'$ of degree~2 contains a hole through $c$ and $c'$ is an NP-complete problem.  
  Let us reduce this problem to testing for a non-good~$P_3$, so consider an instance $G, c, c'$ of Bienstock's problem.  Call $a$ and $b$ the neighbours of $c$, and build a graph $G'$ from $G$ by simply adding the edge $cc'$. 
  
  If $G$ contains a hole $H$ through $c$ and $c'$, then $acb$ is not a good $P_3$ of $G'$, because of the path $H\sm c$. Conversely, if $acb$ is not good, then there exists an $ab$-path $P$ in $G'\sm c$ with an internal vertex $v$ adjacent to $c$. Since $c$ has degree~2 in $G$, we must have $v=c'$.  Hence, $P$ and $c$ form a hole of~$G$ through $c$ and $c'$.     
\end{proof}

Of course, when restricted to (theta, triangle, wac)-free graphs, the problem of testing for good $P_3$'s is much easier (and tractable in polynomial time since as we will see, the treewidth in the class is bounded).

\subsection{Planarity}
\label{ssec:planar}

We now show that our description allows to understand the planar graphs in our class. By the Jordan curve theorem, any hole $C$ in a planar embedding of some graph is a simple closed curve that partitions its complement into a bounded open set and an unbounded open set.  They are respectively the \emph{interior} and the \emph{exterior} of $C$.  
  
\begin{theo}
  \label{th:planar}
  A (theta, triangle, wac)-free graph is planar if and only if it contains no full odd daisy. 
\end{theo}

\begin{proof}
  A full odd daisy is obviously non planar, so a planar graph cannot contain it. The proof of the converse is by induction on the number of vertices of a (theta, triangle, wac)-free graph $G$ that contains no full odd daisy. By Theorem~\ref{th:Decomp}, it is sufficient to consider the following cases. 

  \medskip

  \noindent{\bf Case 1:} $G$ is basic (this is the base case of the induction).
  The graphs $K_1$, $K_2$ and the cube are planar. 
  When $G$ is a daisy, we consider the following planar embedding of $G$: draw the base hole $C$ as a circle. 
  Then draw every petal alternatively in the interior and in the exterior of $C$. 
  Since the daisy is not full and odd, this yields a planar embedding. 

  \medskip
  
  From here on, we consider blocks of decomposition of $G$, that are proper induced subgraphs of $G$, and are therefore planar by the induction hypothesis. 
  
  \medskip

  \noindent{\bf Case 2:} $G$ has a clique separator.
  Since $G$ is triangle-free, the clique separator is on at most~2 vertices and a planar embedding of $G$ is easily recovered from planar embeddings of its blocks of decomposition. 
  
  \medskip
  
  From here on, we consider blocks of decomposition $G_X$ and $G_Y$ with respect to a proper 2-separator $\{a, b\}$ or a proper $P_3$-separator $acb$. 
  Their respective marker are paths $Q$ and $P$, where $P\subseteq X \cup \{a, b\}$ and $Q\subseteq \{a, b\} \cup Y$.  
  We consider planar embeddings $\mathcal G_X$ of $G_X$ and $\mathcal G_Y$ of $G_Y$.
  
  \medskip
  
  \noindent{\bf Case 3:} $G$ has a proper 2-separator with split $(X, \{a, b\}, Y)$. 
  
  Since the internal vertices of $Q$ all have degree~2 (in $G_X$), $Q$ is on the boundary of some face of $\mathcal G_X$.  
  Hence $a$ and $b$ are on the boundary of some face of the embedding $\mathcal G'_X$  obtained from $\mathcal G_X$ by removing the internal vertices of~$Q$.  
  We have a similar property with $G_Y$ and $P$, that yields an embedding~$\mathcal G'_Y$. 
  So a planar embedding of $G$ is obtained by gluing the embeddings $\mathcal G'_X$ and $\mathcal G'_Y$ along $a$ and $b$. 
  
  \medskip
  
  \noindent{\bf Case 4:} $G$ has a proper $P_3$-separator with split $(X, acb, Y)$. We may assume that $G$ is atomic. So $G_X$ and $G_Y$ are also atomic by Lemma~\ref{l:decProperP3SepAtomic}.   
  
  The hole $H_Y = aPbca$ is a face of $\mathcal G_Y$. 
  Indeed, otherwise there are vertices in both the interior and the exterior of $H$. 
  Since $G$ is atomic and every vertex of $H_Y\sm acb$ has degree~2 (in $G_Y$), the interior of $H$ contains a shortest path $R$ from $a$ to $b$ (not going through $c$, but possibly containing neighbours of $c$).  
  Similarly, the exterior of $H$ contains a shortest path $S$ from $a$ to $b$. 
  Hence $R$, $S$ and $H_Y\sm c$ form a theta from $a$ to $b$, a contradiction.
  It follows that $acb$ is on some face of $\mathcal G_Y$. 
  
  We claim that $acb$ is on the boundary of some face of the embedding of $G[X \cup acb]$ obtained by removing the internal vertices of $Q$ in $\mathcal G_X$.  
  Indeed, the paths $P$ and $Q$ induce a hole $H$, that is the rim of a wheel $W$ of $G_X$ centred at~$c$.  
  Up to symmetry suppose that $c$ is in the interior of $H$.  
  For every sector $S = v\dots w$ of $W$ that is contained in $Q$, consider the hole $H_S = cvSwc$. 
  The exterior of $H_S$ is non-empty because of $H$.  
  Only $c$ and at most one of $v$ or $w$ (in case they coincide with $a$ or $b$) may have neighbours in the interior of $H_S$.  
  So, since $G$ is atomic, the interior of $H_S$ is empty, so $H_S$ is in fact the boundary of some face $F_S$ of $\mathcal G_X$. 
  It follows that when the internal vertices of $Q$ are removed, some face $F$ of the embedding contains the union of all faces $F_S$, and $acb$ is on the boundary of $F$, as claimed. 
  
  Since $acb$ is on the boundary of some face of some embeddings of $G[X \cup acb]$ and $G[acb \cup Y]$, a planar embedding of $G$ can be obtained from two such embeddings by gluing them along $acb$.  
\end{proof}

 We have the following corollary. 
 
 \begin{cor}
   Every bipartite (theta, wac)-free graph is planar. 
 \end{cor}

\begin{proof}
  Follows from Theorem~\ref{th:planar} and the fact that a bipartite graph contains no odd daisy. 
\end{proof}

\subsection{Treewidth and cliquewidth}
\label{ssec:tw}

A \emph{tree representation} of a graph $G$ is a pair $(T, \mathcal{X})$ where $T$ is a tree and $\mathcal{X} = \{ X_i \subseteq V(G) : i \in V(T) \}$ is a family of subsets of $V(G)$ (called \emph{bags}) such that:
\begin{enumerate}
    \item\label{i:TWv} For every vertex $v \in V(G)$, there exists at least one bag $X_i$ such that $v \in X_i$.    
    \item\label{i:TWe} For every edge $uv \in E(G)$, there exists at least one bag $X_i$ such that $u, v \in  X_i$.
    \item\label{i:TWc} For every vertex $v \in V(G)$, the set of nodes $\{ i \in V(T) : v \in X_i \}$ induces a connected graph  (so a subtree of $T$).
\end{enumerate}

The \emph{width} of a tree representation $(T, \mathcal{X})$ is defined as $\max_{i \in V(T)} |X_i| - 1$.
The \emph{treewidth} of a graph  $G$ is the minimum width over all possible tree representations of  $G$.

\begin{lemma}
  \label{l:TWconnect}
  If $(T, \mathcal X)$ is a tree representation of some graph $G$ and $H$ is a connected subgraph of $G$, then $\{ i \in V(T) :  X_i \cap H \neq \emptyset \}$ induces a connected subgraph of $T$ (so a subtree of $T$).
\end{lemma}

\begin{proof}
  Clear. 
\end{proof}

The following lemma is often referred to as the Helly property for subtrees of a tree. 

\begin{lemma}[see~\cite{Gavril1974}]
  \label{l:helly}
  If $T_1$, \dots, $T_k$ are subtrees of some tree $T$ such that for all $i, j \in [k]$ we have $V(T_i) \cap V(T_j)\neq \emptyset$, then $V(T_1) \cap \dots \cap V(T_k) \neq \emptyset$. 
\end{lemma} 

When $G$ and $H$ are graphs, a \emph{model for an $H$-minor in a graph $G$} is a family $(X_i)_{i\in V(H)}$ of disjoint connected nonempty subsets of $V(G)$ such that for all distinct $i, j\in V(H)$, if $ij\in E(H)$, then there exists  $u\in X_i$ and $v\in X_j$ such that $uv\in E(G)$.  We also say that \emph{$G$ contains $H$ as a minor}. The following is well known (and a direct consequence of Lemma~\ref{l:helly}). 

\begin{lemma}
  \label{l:minorTW}
    If $G$ and $H$ are graphs such that $G$ contains $H$ as a minor, then $\tw(G) \geq \tw(H)$. 
\end{lemma}

\begin{lemma}
  \label{l:decTW}
  Let $G$ be a graph with a clique separator, a proper 2-separator or a proper $P_3$-separator and let $G_1$, \dots, $G_k$ be its blocks of decomposition with respect to this separator.  We have $\tw(G) = \max (\tw(G_1), \dots, \tw(G_k))$.    
\end{lemma}

\begin{proof}
  We consider the separators one by one. 
  
  \medskip
  
  \noindent{\bf Proper 2-separator}.
  Let $(X, \{a, b\}, Y)$ be a split with respect to a proper 2-separator of $G$. 
  We claim that $G[X\cup \{a, b\}]$ admits a tree representation of width at most $\tw(G_X)$ such that some bag contains $a$ and $b$. 
  Indeed, consider a tree representation $(T, \mathcal X)$ of $G_X$ of width $\tw(G_X)$.  
  Recall that $G_X$ contains a marker path~$Q$. 
  Transform $(T, \mathcal X)$ into a tree representation $(T, \mathcal X')$ of $G[X\cup \{a, b\}]$ by replacing each bag $X_i$ by the bag $X'_i =  \{a\} \cup (X_i \sm (Q\sm b))$ if $X_i\cap (Q\sm b) \neq \emptyset$ and  $X'_i = X_i$ otherwise.
  Conditions~\ref{i:TWv} and~\ref{i:TWe} of the definition of tree representation are clearly satisfied by $(T, \mathcal X')$.  
  Condition~\ref{i:TWc} is also clearly satisfied for all vertices of $G\sm a$.  
  For $a$, it follows from Lemma~\ref{l:TWconnect} applied to $Q\sm b$. 
  By Condition~\ref{i:TWe}, some bag $X_i$ contains both $b$ and its neighbour in $Q$.  
  We then see that $X'_i$ contains $a$ and $b$, and the width of $(T, \mathcal X')$ is at most $\tw(G_X)$, thus proving our claim. 
  
  This proof can be mirrored in $Y$ to obtain a tree representation $(S, \mathcal Y')$ of $G[\{a, b\} \cup Y]$, of width at most $\tw(G_Y)$, and such that some bag $Y_j$ contains $a$ and $b$.  
  Now consider the tree obtained from the union of $S$ and $T$ by adding an edge from $i$ to $j$.  Keeping the same bags, we obtain a tree representation of $G$ of width at most $\max (\tw(G_X), \tw(G_Y))$.  
  The treewidth of $G$ is therefore exactly $\max (\tw(G_X), \tw(G_Y))$ since $G_X$ and $G_Y$ are induced subgraphs of~$G$.
  
  \medskip
  
  \noindent{\bf Proper $P_3$-separator}.
  The proof is similar: we consider a split $(X, acb, Y)$ for some proper $P_3$ separator, and prove the existence of tree representation of $G[X \cup acb]$ of width at most $\tw(G_X)$ and such that some bag contain $a$, $c$ and $b$. To prove its existence, we proceed exactly as above to find a tree representation where some bag contain $a$ and $b$. But some bag also contain $a$ and $c$ (because $ac\in E(G)$), and also $c$ and $b$.  Hence, by Lemma~\ref{l:helly}, some bag contain $a$, $c$ and $b$.  We omit further details that are entirely similar to the previous case. 

  \medskip

  \noindent{\bf Clique separator}.
  It follows directly from Lemma~\ref{l:helly} that in any tree representation of some block of decomposition with respect to a clique separator~$K$, some bag contains~$K$.  The result follows.
\end{proof}

\begin{figure}
    \centering
    \includegraphics[height=5cm]{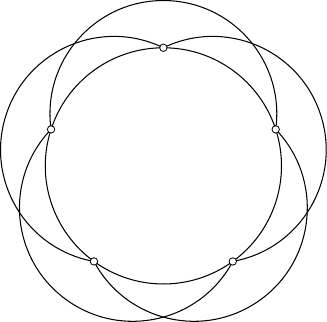}\rule{1cm}{0cm}%
    \includegraphics[height=5.1cm]{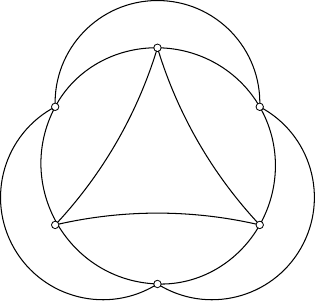}
    \caption{$K_5$ and the octahedron in a daisy fashion}
    \label{fig:K5Octahedron}
\end{figure}

\begin{lemma}
    \label{l:TWdaisy}
    A daisy has treewidth~4 if it is a full $k$-daisy where $k\geq 5$, treewidth~2 if it contains no petal, and treewidth~3 otherwise.   
\end{lemma}

\begin{proof}
  We use notation as in the definition of daisies (except that we denote by $P_i$ the petal centred at $c_{i+1}$, so starting at $c_i$).
  
  Let us describe a tree representation of some full $k$-daisy $G$. 
  We start the construction of the tree $T$ with nodes $t_1$, \dots, $t_k$. The bag corresponding to $t_i$ is $X_i = \{c_i, c_{i+1}, c_{i+2}, c_1, c_2\}$ (subscripts are modulo $k$).  
  To take into account the petal $P_1 = p_1 \dots p_\ell$ centred at $c_{2}$ say, we add to $T$ a path $t_1 t'_1 s_1 \dots s_{\ell-1}$, the bag corresponding to $t'_1$ is $\{c_1, c_2, c_3, p_1\}$ and the bag corresponding to $s_j$ is $\{p_j, p_{j+1}, c_2, c_3\}$.  
  We proceed similarly with all petals.  It is a routine matter to check that this is a tree representation of $G$. The width is~3 when $k=4$, and~4 if $k\geq 5$.    
  
  For a daisy that is not full (but still has at least one petal), suppose up to symmetry that there is no petal centred at $c_1$.  Then we proceed exactly as above, except that we remove from $T$ the vertex $t_k$, and the bag corresponding to $t_i$ is $X_i = \{c_i, c_{i+1}, c_{i+2}, c_1\}$. The width of the representation is then~3. 

  A daisy with no petal is a hole, so it has treewidth~2.  

  For lower bounds, we observe that a daisy with at least one petal contains $K_4$ as a minor, so it has treewidth at least~3 by Lemma~\ref{l:minorTW}. 
  Full daisies with at least five petals are more intricate.  It is useful to notice that contracting each petal into an edge transforms the full daisy with five petals into $K_5$, and the full daisy with six petals into the octahedron, see Figure~\ref{fig:K5Octahedron}.  
  It turns out that $K_5$ and the octahedron are among the four well-known  graphs of treewidth~4 that are minimal with respect to the minor relation see~\cite{Arnborg1990Complexity,Satyanarayana1990LinearTime}. 
  We now check that this observation generalises to larger sizes of full daisies.   
  
  For an odd full $k$-daisy with $k\geq 5$, say $k=2\ell +1$, consider the five sets $A=\{c_1\} \cup V(P_1)$, $B=\{c_2\} \cup V(P_2)$, $C = \{c_3\} \cup V(P_3)$, $D = \cup_{2\leq i \leq \ell} (\{c_{2i}\} \cup V(P_{2i}))$ and $E = \cup_{2\leq i \leq \ell} (\{c_{2i+1}\} \cup V(P_{2i+1}))$. 
  It is a routine matter to check that these fives sets form the model of a $K_5$-minor in $G$, so the treewidth is at least~4 by Lemma~\ref{l:minorTW}. 
  
  For an even full $k$-daisy with $k\geq 6$, say $k=2\ell$, consider the six sets $A_1=\{c_1\} \cup V(P_1)$, $B_1=\{c_2\} \cup V(P_2)$, $C_1 = \{c_3\} \cup V(P_3)$, $A_2 = \{c_4\} \cup V(P_4)$, $B_2 = \cup_{2\leq i \leq \ell-1} ( \{c_{2i+1}\} \cup  V(P_{2i+1}))$ and $C_2 = \cup_{3\leq i \leq \ell} (\{c_{2i}\} \cup V(P_{2i}))$. 
  It is a routine matter to check that these six sets form the model of an octahedron minor in $G$, so the treewidth is at least~4 by Lemma~\ref{l:minorTW}. 
\end{proof}

\begin{lemma}
  \label{l:TWcube}
    The cube has treewidth~3.   
\end{lemma}

\begin{proof}
  The vertex-set of the cube $G$ is $\{v_{i, j} : 1\leq i \leq 2, 1\leq j\leq 4\}$ and $v_{i, j}v_{i', j'} \in E(G)$ if and only if $i\neq i'$ and $j\neq j'$. 
  We consider a tree with five nodes $t_0$, \dots, $t_4$ and four edges $t_0t_1$, \dots, $t_0t_4$.  
  The bag associated to $t_0$ is $X_0 = \{v_{1, 1}, v_{1, 2}, v_{1, 3}, v_{1, 4}\}$.  
  The bag associated to $t_i$ ($1\leq i \leq 4)$ is $X_i = \{v_{2, i}\} \cup X_0 \sm \{v_{1, i}\}$.  It is a routine matter to check that this is a tree representation of the cube of width~3.   
\end{proof}

Theorem~\ref{th:TW} is a corollary of the following. 

\begin{lemma}
  \label{l:tw}
  A (theta, triangle, wac)-free graph has treewidth~4 if it contains a full $k$-daisy where $k\geq 5$, treewidth at most~2 if it is wheel-free, and treewidth~3 otherwise. 
\end{lemma}

\begin{proof}
  If $G$ is wheel-free, its treewidth is at most~2 by Theorem~\ref{cor:us} and Lemma~\ref{l:decTW}.  If $G$ contains a wheel, the conclusion follows from Theorem~\ref{th:Decomp} and Lemmas~\ref{l:TWcube}, \ref{l:TWdaisy} and~\ref{l:decTW} by a trivial induction.
\end{proof}

We now explain how our results improve bounds on the cliquewidth established in~\cite{DBLP:journals/dm/CameronSHV18}.  
We do not need the definition of cliquewidth.  
We only use a theorem from~\cite{DBLP:journals/siamcomp/CorneilR05} stating that the cliquewidth of a graph $G$ is at most $3\times 2^{\tw(G) - 1}$.  It immediately follows that the cliquewidth of any (theta, triangle, wac)-free graph is at most~24. 
As explained in Section~5 of \cite{DBLP:journals/dm/CameronSHV18}, it then easy to prove that for a (theta, cap, $C_4$, even wheel)-free graph, 
 the cliquewidth  is at most 24 and the treewidth at most  $5  \omega(G) -1$. 

\subsection{Testing for a wac}
\label{s:testWac}

\begin{figure}
    \centering
    \includegraphics[height=3.3cm]{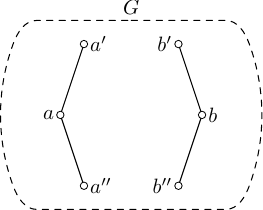}\rule{1cm}{0cm}%
    \includegraphics[height=4.2cm]{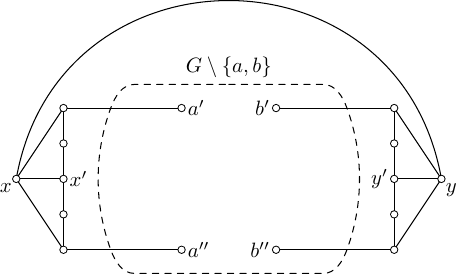}
    \caption{Graphs $G$ and $G'$}
    \label{fig:wacNPC}
\end{figure}

 A \emph{hub} in a graph is a vertex that has at least three neighbours of degree at least~3. Observe that the centre of some wheel is a hub. The proof of the following theorem is similar to the proof from~\cite{diotTaTr:13} that deciding if a graph contains a wheel is NP-complete. 

\begin{theo}
\label{th:wacNPC}
Deciding if an input bipartite graph contains a wac (resp.\ a turtle, a c-wac) is an NP-complete problem. 
\end{theo}

\begin{proof}
  It  proven in~\cite{diotTaTr:13} that the problem $\mathcal P$ of finding a hole through two prescribed vertices $a$ and $b$ of degree~2 in a bipartite hub-free graph is NP-complete.  This is not explained in~\cite{diotTaTr:13}, but by possibly subdividing edges and possibly replacing $b$ by a neighbour of $b$, it may be assumed that $a$ and $b$ are in different sides of the bipartition of $G$.

  We reduce $\mathcal P$ to our problem. So, consider an instance $(G, a, b)$ of $\mathcal P$, see the first graph represented in Figure~\ref{fig:wacNPC}. 
  The neighbours of $a$ and $b$ are $a'$, $a''$, $b'$ and $b''$.  
  Build an instance $G'$ of our problem by deleting $a$ and $b$ and adding vertices to obtain the second graph represented in  Figure~\ref{fig:wacNPC}.
  Since $a$ and $b$ are in different side of the bipartition of $G$, $G'$ is bipartite. 
  Let us prove that $G$ contains a hole through $a$ and $b$ if and only if $G'$ contains a wac (resp.\ a turtle). 
  
  If $G$ contains a hole $H$ through $a$ and $b$, then $G'$ obviously contains a wac (that turns out to be turtle), induced by $H\sm \{a, b\}$ and  $G'\sm G$.   
  
  Conversely, suppose that $G'$ contains a wac (resp.\ a turtle).  Since $x$ and $y$ are the only hubs of $G'$ (because $G$ is hub-free), the wac (resp.\ turtle) must have the form $(C, x, y)$.  The hole $C$ must go through $a'$, $a''$, $b'$  and $b''$, and must therefore contain an $a'b'$-path and a $a''b''$ path, or an $a'b''$-path and a $a''b'$-path.  In either cases, these paths together with $a$ and $b$ form a hole in~$G$.   
  
  We proved that detecting a wac (or a turtle) is an NP-complete problem. To prove that detecting a c-wac is also NP-complete, the proof is the same, except that in $G'$, we erase the edges $xx'$ and $yy'$, and we add in place the edges $xy'$ and $yx'$. 
\end{proof}

\subsection{Algorithms}
\label{ssec:algo}

We now describe an algorithm not relying on the treewidth to build a decomposition tree $T_G$ of some input graph $G$ (this is not a tree representation). 
The nodes of $T_G$ are graphs.  The tree $T_G$ is rooted, and its root is $G$ itself. The following rules apply to construct the rest of $T_G$.

\begin{itemize}
    \item If some node of $T_G$ is a graph $H$ with some clique separator, then for some clique separator $K$ of $H$ with $|K|$ minimal, the children of $H$ are the blocks of decomposition of $H$ with respect to $K$.
    \item If some node of $T_G$ is an atomic graph $H$ with some proper 2-separator, then for some proper 2-separator of $H$ with split $(X, \{a, b\}, Y)$, the children of $H$ are the blocks of decomposition $H_X$ and $H_Y$ of $H$ with respect to $\{a, b\}$.
    \item If some node of $T_G$ is a superatomic graph $H$ with some proper $P_3$-separator, then for some proper $P_3$-separator of $H$ with split $(X, acb, Y)$, the children of $H$ are the blocks of decomposition $H_X$ and $H_Y$ of $H$ with respect to $acb$.
\end{itemize}

Note that for any graph $G$, any decomposition tree $T_G$ is finite because a node of $T_G$ has strictly more vertices than its descendants (if any). In particular a trivial algorithm computes in finite time (but possibly not polynomial) a decomposition tree for any input graph. 

\begin{lemma}
  \label{l:decLeaves}
  A graph $G$ is (theta, triangle, wac)-free (resp.\ (theta, triangle, even wheel)-free, (even hole, triangle)-free,  bipartite (theta, wac)-free)) if and only if all the leaves of $T_G$ are basic graphs from this class.
\end{lemma}

\begin{proof}
  Follows from Lemmas~\ref{l:decClique}, \ref{l:decProper2sep}, \ref{l:decProperP3Sep} and a trivial induction.
\end{proof}

\begin{lemma}
  \label{l:algDec}
  There exists an algorithm whose input is a graph $G$, running in time $O(|V(G)|^3|E(G)|)$, and whose output is one of the following: a triangle of $G$, a clique separator of $G$, a proper 2-separator of $G$, a proper $P_3$-separator of $G$, or a certificate that none of those exists in $G$.    
\end{lemma}

\begin{proof}
  The algorithm simply enumerates by brute force every set of vertices of $G$ on at most 2~vertices, and every set made of one vertex and one edge. Testing if such a set forms a clique is trivial. Testing if it is a separator can be done in time $O(n+m)$  by a simple connectivity check.  Testing all the conditions of being proper is easy. For testing a loose or tight component $X$ of a $P_3$-separator $acb$, just remove all neighbours of $c$ in $X$ and check if some $ab$-path still exists in $X\cup \{a, b\}$.  If it exists, then $X$ is loose, otherwise it is tight.  
\end{proof}

We now introduce notions needed to bound the number of nodes of $T_G$ (provided that $G$ is (theta, triangle, wac)-free).

An \emph{ear} in a graph $G$ is a path $P = x\dots y$ such that $x$ and $y$ both have degree $2$ (in $G$), their respective neighbours $a$ and $b$ not in $P$ are distinct and have a common neighbour $c$, every internal vertex of $P$ either has degree~2 or has degree~3 and is adjacent to $c$, and at least one internal vertex of $P$ is adjacent to $c$. Moreover we require that $P$ is maximal with respect to these properties.  The following lemma justifies the name``ear''. 

\begin{lemma}
If $G'$ is a graph obtained from an atomic graph $G$ by a good ear addition, then the  path added to $G$ to obtain $G'$ is an ear of $G'$. 
\end{lemma}

\begin{proof}
  Suppose that the ear addition consists in adding $P$ to $acb$ (a good $P_3$ of $G$).   
  All conditions of the definition of an ear are clearly satisfied by~$P$ except possibly the maximality of $P$.  
  So suppose that $P$ is not maximal. Up the symmetry between $a$ and $b$, this means that there exists a path $Q = a'\dots b$ containing $P\cup\{a\}$ and such that $a'$ is adjacent to $c$. 
  Moreover, every internal vertex of $Q$ has degree~2 or is adjacent to $c$ and has degree~3.  
  Since $G$ is atomic, some $a'b$-path $R$ exists in $G\sm ca$.  
  Because of the degree conditions, $aQa'Rb$ is a path of $G\sm c$ that contains an internal vertex adjacent to $c$ (namely $a'$), a contradiction to $acb$ being a good $P_3$ of $G$.  
\end{proof}

\begin{lemma}
  \label{l:earDisjoint}
  If $G$ is an atomic graph and not a wheel, then distinct ears of $G$ are disjoint. 
\end{lemma}

\begin{proof}
  Let $P = x\dots y$  and $P' = x'\dots y'$ be distinct ears of $G$, and $a$, $b$, $c$, $a'$, $b'$ and $c'$ be named accordingly as in the definition of ears. Note that $a$ (resp. $b$) has degree at least~3 in $G$, for otherwise $c$ and some internal vertex of $P$ form a clique separator of $G$.
  
  If $P\cap P'\neq \emptyset$, then by the maximality condition in the definition of ears, both $P\sm P'$ and $P'\sm P$ are non-empty. So, since $P'$ is a path, there exists a vertex $u\in P'\sm P$ that has some neighbour in $P\cap P'$. Now notice that, by definition of an ear, the only vertices in $G$ that are adjacent to a vertex of $P$ are $a,~c$ and $b$. So either $u=c$ or, up the symmetry between $a$ and $b$, $u=a$. If $u=c$, then the degree of $c$ in $G$ must be~3, and $c$ is an internal vertex of $P'$. However, all its neighbours have degree at least~3 and its not possible to have two consecutive degree~3 vertices in an ear.
  Hence, $u=a$ and $x\in P'$.
  
  So $a$ is a degree~3 vertex of $P'$, and therefore an internal vertex of $P'$.  Hence, $c\notin P'$ since no two vertices of degree~3 of $P'$ can be adjacent. It follows that $c'=c$.  
  
  Up to the symmetry between $x'$ and $y'$, we may assume that $a \in x' P' x$. If $b\in x' P' a$ or $a'=b$, then because of the degree conditions, all vertices of $P$ and $P'$ (including $a$) have degree~2, or have degree~3 and are adjacent to~$c$.  
  It follows that $P\cup P'$ induces the rim of a wheel $W$ centred at $c$, and only $c$ and $b$ can have neighbours in  $G\sm W$.  
  Hence, either $G=W$, or $cb$ is a clique separator of $G$, a contradiction in both cases.  
  Hence, $b\notin x' P' a$ and $a'\neq b$.  
  It follows that $x'P'a xPy$ is an ear (with $x'a'\in E(G)$, $yb\in E(G)$, $a'\neq b$ and $c$ is a common neighbour of $a'$ and $b$).  
  This contradicts the maximality of~$P$.
\end{proof}

\begin{proof}
  Let $P = x\dots y$  and $P' = x'\dots y'$ be distinct ears of $G$, and $a$, $b$, $c$, $a'$, $b'$ and $c'$ be named accordingly as in the definition of ears. Note that $a$ has degree at least~3 in $G$, for otherwise $c$ and some internal vertex of $P$ form a clique separator of $G$.
  
  If $P\cap P'\neq \emptyset$, then by the maximality condition in the definition of ears, both $P\sm P'$ and $P'\sm P$ are non-empty. So, since $P'$ is a path, there exists a vertex $u\in P'\sm P$ that has some neighbour in $P\cap P'$.  Because of the degree conditions in the definition of ears, up the symmetry between $a$ and $b$, we have $u=a$ and $x\in P'$. So $a$ is a degree~3 vertex of $P'$, and therefore an internal vertex of $P'$.  Hence, $c\notin P'$ since no two vertices of degree~3 of $P'$ can be adjacent. It follows that $c'=c$.  
  
  Up to the symmetry between $x'$ and $y'$, we may assume that $a \in x' P' x$. If $b\in x' P' a$ or $a'=b$, then because of the degree conditions, all vertices of $P$ and $P'$ (including $a$) have degree~2, or have degree~3 and are adjacent to~$c$.  
  It follows that $P\cup P'$ induces the rim of a wheel $W$ centred at $c$, and only $c$ and $b$ can have neighbours in  $G\sm W$.  
  Hence, either $G=W$, or $cb$ is a clique separator of $G$, a contradiction in both cases.  
  Hence, $b\notin x' P' a$ and $a'\neq b$.  
  It follows that $x'P'a xPy$ is an ear (with $x'a'\in E(G)$, $yb\in E(G)$, $a'\neq b$ and $c$ is a common neighbour of $a'$ and $b$).  
  This contradicts the maximality of~$P$.
\end{proof}

 Let us define the \emph{type} $t(G)$ of some graph $G$.

\begin{itemize}
    \item If $G$ is disconnected or $|V(G)| = 1$, then $t(G)=1$.    
    \item If $t(G)\notin [1]$ and $G$ has a vertex separator or $G$ is $K_2$, then $t(G) = 2$.
    \item If $t(G)\notin [2]$ and $G$ has a clique separator or $G$ is a hole or a cube, then $t(G) = 3$.
    \item If $t(G)\notin [3]$ and $G$ has a proper 2-separator or $G$ is a wheel, then $t(G) = 4$. 
    \item If $t(G)\notin [4]$, then $t(G) = 5$.
\end{itemize}

\begin{lemma}
  \label{l:f-child}
  Let $G$ be a (theta, triangle, wac)-free graph. If $H'$ is a descendant of $H$ in $T(G)$, then $t(H') \geq t(H)$.
\end{lemma}

\begin{proof}
  It is enough to prove the property when $H'$ is a child of $H$ (the rest follows by induction). 
  
  If $t(H) = 1$, then clearly $t(H')\geq t(H)$. 
 
  If $t(H) = 2$, then $H$ has a vertex separator ($G = K_2$ is impossible, since $H$ is not a leaf). Hence $H'$ is clearly connected and contains at least two vertices, so $t(H') \geq 2$. 
 
  If $t(H) = 3$, then $H$ has a $K_2$ separator $K$ (since $H$ is not a leaf). Hence $t(H') \geq 3$. 
  Indeed, the node $H'$ cannot be a clique because it properly contains $K$. 
  Also, a vertex separator of $H'$ would be a vertex separator of $H$, so $H'$ has no vertex separator.  
 
  If $t(H) = 4$, then $H$ has a proper 2-separator $\{a, b\}$ (since $H$ is not a leaf). By Lemma~\ref{l:decProper2sepAtomic}, $H'$ is atomic. 
  Since $a, b \in H'$, $H'$ cannot be a clique.  
  Also, $H'$ cannot be a hole because $\{a, b\}$ is proper. Hence $t(H') \geq 4$. 
 
  If $t(H) = 5$, then $H$ has a proper $P_3$-separator $acb$ (since $H$ is not a leaf). By Lemma~\ref{l:decProperP3Sep-2Sep}, $H'$ is superatomic. Since $a, b \in H'$, $H'$ cannot be a clique.  Obviously, $H'$ cannot be a cube (because the marker path has vertices of degree~2).
  Also, $H'$ cannot be a hole or a wheel because $acb$ is proper (a block being a hole or a wheel would imply that some side is a path). 
  Hence $t(H') \geq 5$.
\end{proof}

A \emph{chunk} of a graph is either an ear, or a vertex of degree at least~3 not contained in an ear. Observe that a daisy with $\ell\geq 1$ petals contains $2\ell + 2$ chunks: $\ell$ of them are the petals (that are ears, see Figure~\ref{fig:earPetal}), the rest are vertices of degree~3 or~4 of the base hole.  We define the following functions of some graph $G$.

\begin{figure}[hbt]
    \centering
    \includegraphics[width=4cm]{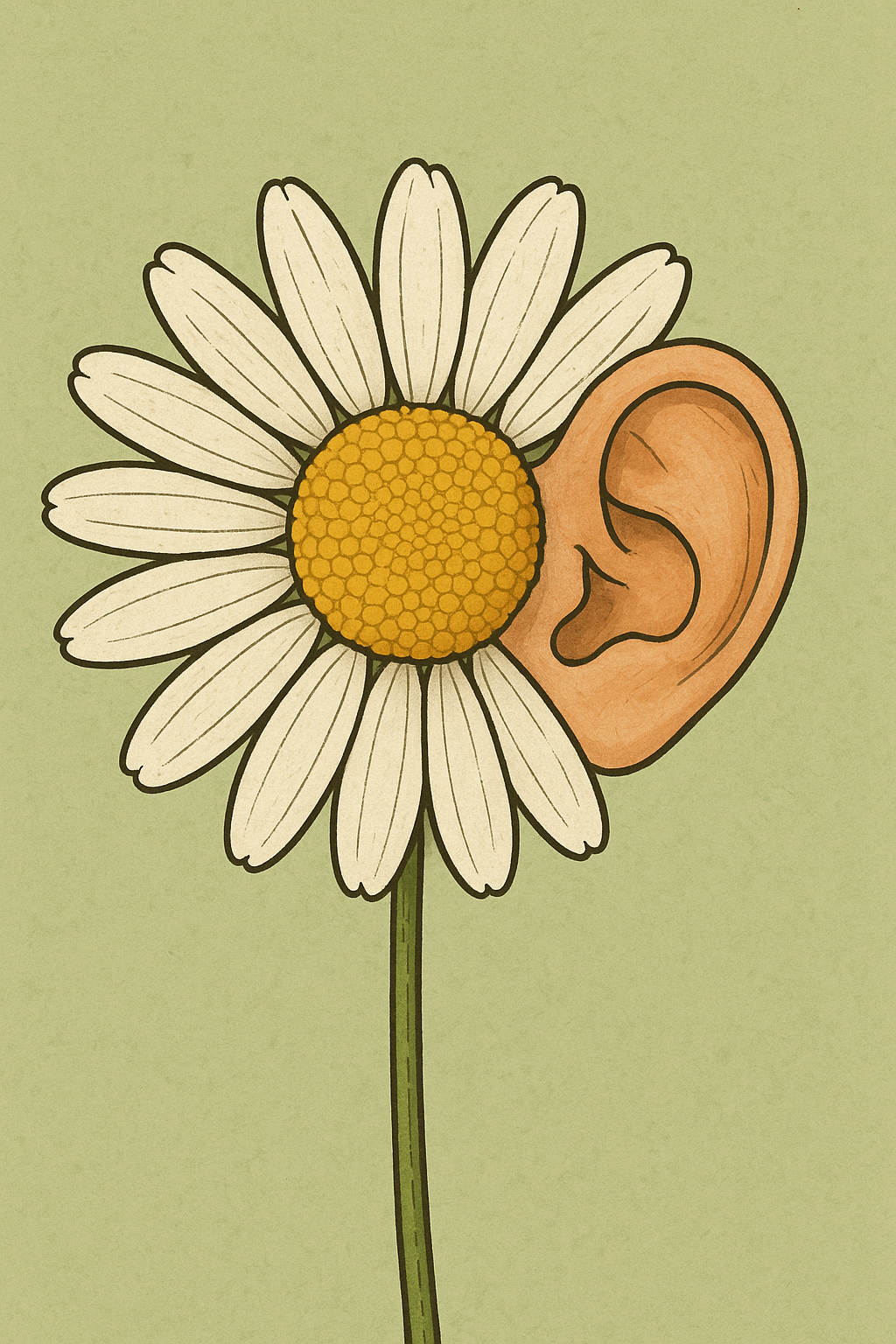}
    \caption{Should daisies have ears or petals?}
    \label{fig:earPetal}
\end{figure}

\begin{itemize}
    \item $f_1(G) = |V(G)|$
    \item $f_2(G) = |V(G)| - 1$
    \item $f_3(G) = |V(G)| - 2$
    \item $f_4(G) = |\{v\in V(G) : \text{ $v$ has degree at least 3}\}| - 3$
    \item $f_5(G) = |\{X \subseteq V(G) : \text{ $X$ is a chunk of $G$}\}| - 5$
    \item $f(G) = f_{t(G)}(G)$.  
\end{itemize}

\begin{lemma}
  \label{l:f-monotonous}
  For every graph $G$, $$|V(G)| \geq f_1(G) \geq f_2(G) \geq f_3(G) \geq f_4(G).$$

  Moreover, if $G$ is atomic and distinct from a wheel, then $f_4(G) \geq f_5(G)$.
\end{lemma}

\begin{proof}
  The first inequalities are clear. 
  The last inequality comes from the fact that under our assumptions, the chunks of $G$ are disjoint by Lemma~\ref{l:earDisjoint} and each chunk contains a vertex of degree at least~3.
\end{proof}

\begin{lemma}
  \label{l:share}
  If $G$ is a (theta, triangle, wac)-free graph, then every node $H$ of $T_G$ satisfies  the following. 
  \begin{itemize}
    \item If $H$ is a leaf of $T_G$, then $f(H) \geq 1$. 
    \item If $H$ is has children $H_1$, \dots, $H_k$ in $T_G$, then $$f(H) \geq f(H_1) + \cdots + f(H_k).$$ 
  \end{itemize}
\end{lemma}

\begin{proof}
  Suppose first that $H$ is a leaf of $T_G$. By Lemma~\ref{l:decLeaves}, $H$ is basic. If $H=K_1$, then $t(H) = f(H) = f_1(H) = 1$.  If $H=K_2$, then $t(H)=2$ and $f(H) = f_2(H) = 1$. If $H$ is a hole or the cube, then $t(H) = 3$ and $f(H) = f_3(H) \geq 2$.  If $H$ is wheel, then $t(H) = 4$ and $f(H) = f_4(H)$. Since $H$ contains at least four vertices of degree at least~3, $f(H) \geq 1$. If $H$ is none of the graphs already addressed, then $H$ is a daisy with $\ell\geq 2$ ears, so it contains $2\ell + 2$ chunks.  Hence, $t(H) = 5$ and $f(H) = f_5(H) = 2\ell + 2 -5 \geq 1$.  In all cases, $f(H) \geq 1$. 
  
  Suppose now that $H$ is not a leaf of $T_G$.  In particular, $H$ is not a wheel. By Lemmas~\ref{l:f-child} and~\ref{l:f-monotonous}, it is enough to check that $$f_{t(H)}(H) \geq f_{t(H)}(H_1) + \cdots + f_{t(H)}(H_k).$$ 
  
  If $t(H)=1$, then $H$ is disconnected and its children in $T_G$ are its connected components, so $$f_1(H) = |V(H)| = |V(H_1)| + \cdots + |V(H_k)| = f_1(H_1) + \cdots + f_1(H_k).$$
  
  If $t(H)=2$ or $t(H)=3$, the proof is similar. Note that the -1 or -2 in the definition of $f$ avoids double counting the (1 or 2) vertices that are in all blocks of decomposition. 
   
  If $t(H)=4$, then let $(X, \{a, b\}, Y)$ a split with respect to a proper 2-separator of $G$.  Let $x$ be the number of vertices of $X$ that have degree at least~3 (in $G$ and equivalently in $G_X$). Let $y$ be defined similarly for $Y$. 
  If both $a$ and $b$ have degree~2 in $G$ (and therefore in $G_X$ and $G_Y$), we have $$f_4(G) = x+y -3 \geq (x-3) + (y-3) = f_4(G_X) + f_4(G_Y).$$ So, we may assume that at least one of $a$ or $b$ has degree at least~3 in $G$. 
  Hence we have $$f_4(G) \geq x+y-2 = (x-1) + (y-1) \geq f_4(G_X) + f_4(G_Y).$$
  
  If $t(H)=5$, then let $(X, acb, Y)$ be a split with respect to a proper $P_3$-separator of $G$.  Since $acb$ is proper, $c$ has degree at least~4 in $G$. So, $c$ is a chunk of $G$. Also, no chunk of $G$ contains both $a$ and $b$, because otherwise such a chunk $A$ would contain an $aXb$-path or an $aYb$-path  that would be a component of $G \sm acb$ (because of the degree condition in the definition of an ear), a contradiction to $acb$ being proper. 
  
  We claim that every chunk of $G_X$ that is included in $X$ is a chunk of~$G$.  Otherwise, consider a  counter-example $A$. Since $A$ is not a chunk of $G$, there exists an ear $B$ of $G$ such that $A \subsetneq B$. As already noted, we have $c\notin B$ since $c$ is a chunk of $G$. We have $a\in B$ or $b\in B$ (say $a\in B$ up to symmetry) for otherwise, $B\subseteq X$, so $B$ is an ear of $G_X$, a contradiction to Lemma~\ref{l:earDisjoint}.  So, $(B\sm Y) \cup Q$ is contained in an ear of $G_X$, a contradiction to $A$ being a chunk of $G_X$.  We proceed similarly with $G_Y$. 
  
  Let $x$  be the number of chunks of $G_X$ that are included in $X$. Let $y$ be defined similarly with  $Y$. 
  We have $f_5(G) \geq x+y+3-5 = x + y - 2$, because $a$, $c$ and $b$ are in distinct chunks of $G$ that are neither included in $X$ nor $Y$.  

  In $G_X$, the interior of the marker path $Q$ is included in one chunk of $G_X$ (because it is contained in some ear). So, the number of chunks in $G_X$ is at most $x+4$ (because of $Q$, $a$, $b$ and $c$) and $f_5(G_X) \leq x+4 - 5 = x - 1$.  
  A similar inequality holds for $G_Y$, because the marker path of $G_Y$ is included in at most one chunk. 
  So, $$f_5(G) \geq x + y - 2 \geq (x - 1) + (y-1) \geq f_5(G_X) + f_5(G_Y).$$
\end{proof}

\begin{lemma}
  \label{l:TGnNodes}
  If $G$ is (theta, triangle, wac)-free, then $T_G$ has at most $2|V(G)| - 1$ nodes.
\end{lemma}

\begin{proof}
  By Lemma~\ref{l:share}, $T_G$ has at most $f(G)$ leaves, so at most $|V(G)|$ leaves. It follows that it has at most $2|V(G)| - 1$ nodes. 
\end{proof}

\begin{theo}
  \label{th:algo}
  There exists an algorithm that decides in time $O(|V(G)|^4|E(G)|)$ whether an input graph $G$ is (theta, triangle, wac)-free (resp.\ (theta, triangle, even wheel)-free, (even hole, triangle)-free,  bipartite (theta, wac)-free). Moreover, when it is, the algorithm outputs a tree decomposition of $G$ of width $\tw(G)$ (in particular, of width at most~4).
\end{theo}

\begin{proof}
  The algorithm builds a decomposition tree, but stops the process whenever $2|V(G)|$ nodes are constructed.  This takes time at most $O(|V(G)|^4|E(G)|)$ by Lemma~\ref{l:algDec}.  
  
  If exactly $2|V(G)|$ nodes are constructed, the algorithm outputs that $G$ is not (theta, triangle, wac)-free and stops, which is correct by Lemma~\ref{l:TGnNodes}. 
  
  If at least one leaf $H$ of $T_G$ is not basic (which is easy to test in time $O(|V(H)|^3)$), the algorithm outputs that $G$ is not (theta, triangle, wac)-free and stops, which is correct by Lemma~\ref{l:decLeaves}. 

  Otherwise, the algorithm outputs that $G$ is (theta, triangle, wac)-free.  By Lemmas~\ref{l:decClique}, \ref{l:decProper2sep}, \ref{l:decProperP3Sep}, a further analysis of the leaves can determine whether $G$ is (theta, triangle, even wheel)-free, (even hole, triangle)-free,  or bipartite (theta, wac)-free. 
  Also, the proofs of Lemmas~\ref{l:decTW}, \ref{l:TWdaisy} and \ref{l:TWcube} shows how to actually compute from $T_G$ a tree representation of $G$ of width $\tw(G)$. 
\end{proof}

\subsection{Structure theorem}
\label{ssec:struct}

We insist that Theorem~\ref{th:Decomp} is not just a decomposition theorem but a structure theorem (though there is no formal definition of these two notions).  
A structure theorem should allow building all graphs in some class (and only them) starting from basic blocks by gluing previously constructed graphs along some simple operations. 
So, the basic graphs need no further explanation. 
It is a routine matter to show how the three kinds of separator that we use can be reversed into operations.  We do it below (it is a bit lengthy because of the "proper" conditions in the $P_3$-separator).

If $k\geq 2$ disjoint graphs $G_1$, \dots, $G_k$  are such that each $G_i$ contains a clique $K_i$ on $\ell \leq 2$ vertices, the graph $G$ obtained from the $G_i$'s by gluing them at $K_i$ is said to be \emph{obtained from the $G_i$'s by gluing along a clique} (if $\ell=0$, $G$ is just the disjoint union of the $G_i$'s). 

Let $G_X$ and $G_Y$ be two disjoint graphs. Suppose that $G_X$ contains an $a_Xb_X$-path $Q$ of length at least~2, and $G_Y$ contains an  $a_Yb_Y$-path $P$ of length at least~2.  
Suppose that the internal vertices of $P$ and $Q$ all have degree~2. Let $G$ be obtained from $G_X$ and $G_Y$ by identifying $a_X$ and $a_Y$, identifying $b_X$ and $b_Y$, and deleting the interior of $P$ and $Q$.  
Then $G$ is said to be \emph{obtained from $G_X$ and $G_Y$ by gluing along a proper 2-separator}. 

Let $G_X$ and $G_Y$ be two disjoint graphs. Suppose that $G_X$ contains a path $a_Xc_Xb_X$, an  $a_Xb_X$-path $Q$ of length at least~2, not containing $c_X$ and whose interior contains at least one neighbour of $c_X$, and an $a_Xb_X$-path not containing $c_X$ or a neighbour of $c_X$.  
Suppose that the internal vertices of  $Q$ all have degree~2, or have degree~3 and are adjacent to $c_X$.  
Suppose that $G_Y$ contains a path $a_Yc_Yb_Y$, an  $a_Yb_Y$-path $P$ of length at least~2 not containing $c_Y$ or a neighbour of $c_Y$, and every $a_Yb_Y$-path in $G_Y$ apart from $P$ and $a_Yc_Yb_Y$  contains a neighbour of $c_Y$.  
Suppose that the internal vertices of $P$ all have degree~2.  
Let $G$ be obtained from $G_X$ and $G_Y$ by identifying $a_X$ and $a_Y$, identifying $c_X$ and $c_Y$,  identifying $b_X$ and $b_Y$, and deleting the interior of $P$ and $Q$.  
Then $G$ is said to be \emph{obtained from $G_X$ and $G_Y$ by gluing along a proper $P_3$-separator}.

\begin{theo}
  \label{th:struct}
  A graph is (theta, triangle, wac)-free graph if and only if it can be obtained from basic graphs by repeatedly gluing along cliques, proper 2-separators and proper $P_3$-separators.
\end{theo}

\begin{proof}
  A (theta, triangle, wac)-free graph can be obtained from basic graphs by the operations because of Theorem~\ref{th:Decomp} (and an easy induction). For the converse, we need variants of  Lemmas~\ref{l:decClique}, \ref{l:decProper2sep} and~\ref{l:decProperP3Sep}, where  $G_X$ and $G_Y$ are not necessarily induced subgraphs of $G$, because the marker paths might not have the right length or the right number of neighbours of $c$.  We omit the details that are easy to check.    
\end{proof}

A similar theorem can be easily obtained for (theta, triangle, even wheel)-free graphs, (theta, triangle, even hole)-free graphs and bipartite (theta, wac)-free graphs (not a word in the statement of the theorem has to be changed, just the definition of the basic graphs is different, see Lemma~\ref{l:daisyEHF}).  Also, by Lemmas~\ref{l:decProper2sepAtomic}, \ref{l:decProperP3SepAtomic} and~\ref{l:decProperP3Sep-2Sep}, one may obtain structural descriptions of superatomic graphs of the classes (by restricting the operations to gluing along $P_3$ separators), and of atomic graphs of the classes (by restricting the operations to gluing along 2-separators and extending the basic graphs to superatomic graphs of the class).

\subsection{Open questions}
\label{ssec:conj}

A natural question is whether Theorem~\ref{th:Decomp} is best possible.  Because of Lemma~\ref{l:looseSide}, it is tempting to define a loose component $X$ of $G\sm acb$ as one such that every $aXb$-path has no internal vertex adjacent to $c$.  This can be rephrased as ``$acb$ is a good $P_3$ of $G[X \cup acb]$''.  
All our theorems would stay true and we would have a more precise conclusion of Theorem~\ref{th:Decomp}. 
However, by Theorem~\ref{th:goodNPC}, it would be coNP-complete to decide if a graph contains a proper $P_3$-separator, so the theorem would be less practical. 

We wonder whether proper 2-separators and proper $P_3$-separators preserve bounded cliquewidth. We propose the following conjecture.  

\begin{conj}
  There is a function $f$ with the following property. Let $\mathcal C$ and $\mathcal B$ be two classes of graphs closed under taking induced subgraphs and such that every 2-connected graph from $\mathcal C$ is either in $\mathcal B$ or has a proper 2-separator or a proper $P_3$-separator.  Moreover  for some constant $a$, every graph in $\mathcal B$ has cliquewidth at most $a$. Then, every graph in $\mathcal C$ has cliquewidth at most $f(a)$. 
\end{conj}

We wonder whether it is possible to describe the most general (theta, triangle)-free graph with no separator as defined in the present work.  We would need to combine in the description at least the (theta, triangle)-free layered wheels from \cite{DBLP:journals/jgt/SintiariT21} and the daisies.  An intermediate step could be to study (theta, triangle, c-wac)-free graphs. 

\section*{Acknowledgement}

Thanks to Kristina Vu\v skovi\'c for suggesting several improvements to the original manuscript. 

Both authors are supported by Projet ANR GODASse, Projet-ANR-24-CE48-4377.


\end{document}